\documentclass[a4paper]{article}
\makeatletter

\newcommand{\Rmnum}[1]{\expandafter\@slowromancap\romannumeral #1@}
\makeatother
\usepackage{graphicx}
\usepackage{amsfonts}
\usepackage{amsmath}
\usepackage{amssymb}
\usepackage{fancyhdr}
\usepackage{titlesec}
\usepackage{indentfirst}
\usepackage{tikz}
\usepackage{tikz-cd}
\usepackage{booktabs}
\usepackage{verbatim}
\usepackage{color}
\usepackage{amsthm}
\usepackage{subfigure}
\usepackage{abstract}
\usepackage{multirow}
\usepackage{fullpage}
\usepackage{hyperref}
\usepackage[margin=0pt]{geometry}
\usepackage{lipsum}
\usepackage{layout}
\usepackage{mathtools}
\usepackage[page,header]{appendix}
\usepackage{titletoc}
\usepackage{newfloat}
\usepackage{csquotes}
\usepackage{mathrsfs} 
\numberwithin{equation}{section}
\newtheorem{theorem}{Theorem}[section]
\newtheorem{lemma}[theorem]{Lemma}
\newtheorem{example}[theorem]{Example}

\newtheorem{question}[theorem]{Question}
\newtheorem{definition}[theorem]{Definition}
\newtheorem{remark}[theorem]{Remark}

\newtheorem{conjecture}{Conjecture}

\newcommand{\RNum}[1]{\uppercase\expandafter{\romannumeral #1\relax}}

\DeclareMathOperator{\Gal}{Gal}
\DeclareMathOperator{\ind}{ind}

\DeclareMathOperator{\Hom}{Hom}

\DeclareMathOperator{\Stab}{Stab}

\DeclareMathOperator{\Disc}{Disc}

\DeclareMathOperator{\Aut}{Aut}

\DeclareMathOperator{\Ker}{Ker}
\DeclareMathOperator{\ord}{ord}

\DeclareMathOperator{\Cl}{Cl}

\DeclareMathOperator{\Res}{Res}

\usepackage[OT2,T1]{fontenc}
\DeclareSymbolFont{cyrletters}{OT2}{wncyr}{m}{n}
\DeclareMathSymbol{\Sha}{\mathalpha}{cyrletters}{"58}

\newcommand{\charf}{\operatorname{char}}

\newcommand{\C}{\ensuremath{{\mathbb{C}}}}
\newcommand{\Z}{\ensuremath{{\mathbb{Z}}}}

\newcommand{\Q}{\ensuremath{{\mathbb{Q}}}}
\newcommand{\R}{\ensuremath{{\mathbb{R}}}}
\newcommand{\F}{\ensuremath{{\mathbb{F}}}}

\DeclareFloatingEnvironment[fileext=dia,within=section]{diagram}

\usepackage{pdfsync}

\begin{document}
\title{Counterexamples for T\"urkelli's Modification on Malle's Conjecture}

\author{Jiuya Wang}

\newcommand{\Addresses}{{
		\bigskip
		\footnotesize		
		Jiuya Wang, \textsc{Department of Mathematics, University of Georgia, GA 30602, USA
		}\par\nopagebreak
		\textit{E-mail address}: \texttt{jiuya.wang@uga.edu}	
	}}
\maketitle	
	\begin{abstract}
We give counterexamples for the modification on Malle's Conjecture given by T\"urkelli. T\"urkelli's modification on Malle's conjecture is inspired by an analogue of Malle's conjecture over a function field. As a consequence, our counterexamples demonstrate that the $b$ constant can differ between function fields and number fields. We also show that Kl\"uners' counterexamples give counterexamples for a natural extension of Malle's conjecture to counting number fields by product of ramified primes. We then propose a refined version of Malle's conjecture which implies a new conjectural value for the constant $b$ for number fields. 
	\end{abstract}
	
\bf Key words. \normalfont Malle's conjecture, roots of unity, cyclotomic extension, embedding problem
\pagenumbering{arabic}	

\section{Introduction}
\subsection{Malle's Conjecture}
It is a standard result in algebraic number theory that there are finitely many number fields with bounded discriminant. It is then natural to ask how many number fields there are with bounded discriminant. Malle \cite{Mal02, Mal04} gives a conjectural asymptotic answer for this question. For each number field $F/k$ with degree $n$, we define the permutation Galois group $\Gal(F/k)$ to be the image of $G_{k}$ in $S_n$ induced by $G_{k}$ action on embeddings of $F$ into $\bar{k}$. Given a transitive group $G\subset S_n$, we denote $N_k(G, X)$ to be the number of continuous surjections $\rho: G_k\to G$ where the associated degree $n$ number field $F$ fixed by $\Stab(1)\subset G$ has $\Disc(F/k)\le X$. The conjecture states:
\begin{conjecture}[Malle's Conjecture over Number Fields, \cite{Mal02,Mal04}]\label{conj:Malle}
	Given a number field $k$ and a transitive permutation group $G\subset S_n$. There exist positive constants $c(G,k)$, $a(G)\in \Z$ and $b_M(G,k)\in \Z$ such that 
	\begin{equation}
		N_k(G, X) \sim  c(G,k) X^{1/a(G)} \log^{b_M(G,k)-1} X.
	\end{equation}
\end{conjecture} 
\noindent Malle also gives a precise conjectural value for $a(G)$ in \cite{Mal02} and for $b_M(G,k)$ in \cite{Mal04}, see Section \ref{sec:b} for a precise description of Malle's proposed constants.  Here we use the subscript in $b_M(G,k)$ to distinguish it from the true value $b(G,k)$ for powers of $\log X$. 

Progress has been made towards this conjecture \cite{Wri89,DH71,DW88,Bha05,Bha10,Klu12,Kluthesis,CDyDO02,JW17,MTTW,BW08,BF,KP21,KWell,fouvry2021malle}. The integer $a(G)$ has been widely believed to be correct. In all cases towards Conjecture \ref{conj:Malle} where the asymptotic distribution for $N_k(G,X)$ is determined, the true value matches the conjectural $a(G)$. Conjecture \ref{conj:Malle} is also sometimes termed as \textit{strong Malle's Conjecture}, in contrast to the \textit{weak Malle's Conjecture} where $N_k(G, X)$ is predicted to be bounded between $X^{1/a(G)}$ and $X^{1/a(G)+\epsilon}$ asymptotically. Recently, work of Ellenberg-Tran-Westerland \cite{ETW} proves the upper bound in the weak Malle's conjecture for every permutation group $G$ over a global function field $k = \F_q(t)$ with large $q$ relatively prime to $|G|$, which gives a strong evidence towards the validity of $a(G)$. In \cite{KWell} Kl\"uners and the author show that the upper bound in weak Malle's conjecture over number fields is equivalent to the discriminant multiplicity conjecture, and the latter is proved for all nilpotent permutation groups. For solvable groups, the discriminant multiplicity conjecture is simply equivalent to $\ell$-torsion conjecture, and the latter is shown to be a consequence of a much weakened version of Cohen-Lenstra type heuristics \cite{pierce2019conjecture}. So we have firm belief in $a(G)$.

The integer $b(G,k)$ is more mysterious. In 2005, Kl\"uners \cite{Klu} gives counterexamples to Conjecture \ref{conj:Malle} by noticing that certain intermediate cyclotomic extensions can contribute larger exponents for $\log X$ than Malle's prediction $b_M(G,k)$. Among these counterexamples of similar spirit, the most famous one is the wreath product $C_3\wr C_2$. Indeed the number of $C_3\wr C_2$-extensions containing the cyclotomic field $\Q(\mu_3)$ already contributes higher powers of $\log X$ than $b_M(G,\Q)$. Essentially, due to the presence of intermediate cyclotomic extensions, one can follow the same construction to show that Conjecture \ref{conj:Malle} is inconsistent with itself in general, without proving any distribution, see Remark \ref{rmk:malle-inconsist}. We do not discuss the constant $c(G,X)$ in this paper, please see recent preprint \cite{loughran2024malle} for discussions. 

\subsection{T\"urkelli's Modification: inspiration and comparison with function fields}
Like many problems in number theory, we can study the counterparts over global function fields. Conjecture \ref{conj:Malle} over global function fields can be stated in a similar way (see Remark \ref{rmk:theta-function-field} for explanations on why we do not conjecture an asymptotic distribution). We state the Malle's conjecture for global function fields, exactly as \cite[Conjecture 1.1]{turkelli2015connected} formulates it (we say $f(X) = \Theta (g(X))$ if $C_1 g(X) \le f(X) \le C_2 g(X)$ for sufficiently large $X$):
\begin{conjecture}[Malle's Conjecture over Function Field]\label{conj:Malle-function-field}
	Given a global function field $Q$ and a transitive permutation group $G\subset S_n$ with $(|G|, \charf(Q)) = 1$, define $a(G)$ and $b_M(G,Q)$ as in Conjecture \ref{conj:Malle}. Then
	\begin{equation}
		N_Q(G,X) = \Theta(X^{1/a(G)} \ln^{b_M(G,Q)-1} X).
	\end{equation}
\end{conjecture}
\noindent Kl\"uners' counterexamples also hold over global function fields if we allow constant extensions contained in $G$-extensions. 

In order to accommodate these counterexamples, T\"urkelli in \cite{turkelli2015connected} gives a modification of Conjecture \ref{conj:Malle} by proposing a new $b$-constant $b_T(G,Q)$ for both function fields and number fields, see Definition \ref{def:tb} in Section \ref{sec:tb} for the description. 
\begin{conjecture}[T\"urkelli's Modification \cite{turkelli2015connected}]\label{conj:T\"urkelli}
	Given a transitive permutation group $G\subset S_n$ and a global field $Q$ with $(|G|, \charf(Q))=1$, 
	\begin{equation}
		b(G, Q) = b_T(G,Q).
	\end{equation}
\end{conjecture}
It is based on an extension of Ellenberg-Venkatesh's heuristic argument, where $b(G,Q)$ is related to the number of geometrically connected components of Hurwitz spaces. In \cite{ellenberg2005counting} both $a(G)$ and $b_M(G,Q)$ are shown to match the true counting function when enumerating extensions without constant fields (under the heuristic). Similar proposals are also proposed in \cite{darda2023torsors}. T\"urkelli's new input is to consider $G$-extensions $L/Q$ with the fixed subfield $L^N$, for some normal subgroup $N\in G$ with its quotient $G/N$ cyclic, being exactly the maximal constant extensions contained in $G$. In \cite{turkelli2015connected}, T\"urkelli computed the number of rational components of certain Hurwitz spaces and used this to give an extension of the heuristics in \cite{ellenberg2005counting}, which leads to a conjectural distribution for such $L$, given explicity in \cite{turkelli2015connected}. It is then natural that on the function field side the modification is to take the sum over the finitely many possible constant extensions. To translate the problem to number fields, T\"urkelli considers the sum of his heuristic distribution over all subfields of $Q(\mu_{\infty})$ with Galois group a quotient of $G$.
 
In this paper, we give counterexamples for Conjecture \ref{conj:T\"urkelli} in Theorem \ref{thm:example-1}. We demonstrate the key idea via a simple example:
\begin{example}\label{exa:1}
	Let $G = C_{3}\wr C_4\subset S_{12}$ and $\gcd(q, |G|) =1$ and $q$ large enough compared to $G$. We have
	$$b_T(G, \Q) = b(G, \F_q(t)) = 2 ,\quad\quad  b(G, \Q) = b_M(G, \Q) = 1.$$
\end{example}
\noindent Replacing $C_2$ with $C_4$ in Kl\"uners' counterexample has forbidden the existence of the cyclotomic field $\Q(\mu_3)$ as an intermediate extension over $\Q$. We now make a couple of comments on this example. Firstly, it is probably surprising that in this example the prediction of Malle is correct whereas the modification of T\"urkelli is not! Secondly and more importantly, even when a number field and a global function field have the same relevant cyclotomic extension $\Gal(Q(\mu_{d})/Q)$ for $d$ determined by $G$, it can happen that the $b$-constants are different! Thirdly, when $\Gal(Q(\mu_{d})/Q)$ are the same, it seems that there exist more field extensions on the function field side. 

We now demonstrate the reasoning behind the example and the last comment more carefully. Given a finite group $G$ and $G/N$ a quotient group, the study of whether a particular $G/N$-extension can be embedded into a $G$-extension is called \emph{the embedding problem}. It is a study with rich history and theory, and historically played a central role in solving the inverse Galois problem for solvable groups. A particular embedding problem becomes necessary in studying $b(G,k)$ in Conjecture \ref{conj:Malle}: given a fixed cyclotomic extension, i.e., $k(\mu_{n})$ with $\Gal(k(\mu_n)/k)$ being a quotient of $G$, can $k(\mu_{n})$ be embedded into a $G$-extension. We formulate precisely the following question:
\begin{question}[Embedding Cyclotomic Extensions]\label{que:embedding-cyclotomic}
Let $Q$ be a global field. Given a surjective group homomorphism $\pi: G\to B$ and a cyclotomic $B$-extension $F/Q$ (equivalently a surjective continous group homomorphism $\phi: G_{Q} \to B$ that factors through $G_{Q}/G_F$), does there exist a surjective $\tilde{\phi}: G_Q \to G$ such that $\pi \circ\tilde{\phi}  = \phi$?
		\begin{center}\label{diag:embedding}
		\begin{tikzcd}	
			&  &  & G_Q \arrow[ld, "\tilde{\phi}" ', dashed]\arrow[d,"\phi",two heads] &  \\		
			0\arrow{r} & \Ker(\pi) \arrow{r} & G \arrow[r, "\pi"] & B \arrow{r} & 0 \\
		\end{tikzcd}
	\end{center}
\end{question}
\noindent In Example \ref{exa:1}, we can easily see that a $C_4$-extension cannot contain $\Q(\mu_3)$ because one encounters local obstructions at both $p=3$ and $p=\infty$. It seems quite difficult to solve this problem in full generality. We will discuss some cases in Section \ref{sec:embed}, which suffices for proving the following theorem, giving an infinite family of examples where $b(G,\Q)$ is bounded between $b_T(G, \Q)$ and $b_M(G, \Q)$.  

\begin{theorem}\label{thm:example-1}
Let $\ell$ be an odd prime number and $d = \prod_i p_i^{r_i} \neq 2$ where $p_i$ are all prime numbers. Let $G = C_{\ell} \wr C_{d}\subset S_{\ell d}$ with $(|G|, \charf(Q)) = 1$, $\Gal(\F_q(t)(\mu_{\ell})/\F_q(t)) \simeq \Gal(\Q(\mu_{\ell})/\Q)$ and $q$ large enough compared to $G$. Denote $\gcd(d, \ell-1) = \prod_i p_i^{s_i}$. Let $s=\text{val}_2(\ell-1) -1$ when $\text{val}_2(d) > \text{val}_2(\ell-1)$ and $s = 0$ otherwise. Then
$$b_T(G, \Q) = b(G, \F_q(t)) =  \prod_i p_i^{s_i},  \quad b(G, \Q) = \prod_{i, r_i = s_i} p_i^{s_i} \cdot 2^s , \quad b_M(G, \Q) = 1.$$
As a consequence, there exists $G\subset S_n$ such that $\Gal(\F_q(t)(\mu_{\ell})/\F_q(t) )\simeq \Gal(\Q(\mu_{\ell})/\Q)$ and $b(G, \F_q(t)) > b(G, \Q)$.
\end{theorem}
\noindent We exclude $d=2$ only because $b(G, \Q)$ is not proved over number fields currently due to the lack of good $\ell$-torsion bound. One can construct more examples along this line, either in wreath products or non-wreath products. We don't try to expand in that direction here. Theorem \ref{thm:example-1} provides an infinite family of examples that are not complicated as groups, and also indicates the robustness in the comments we make after Example \ref{exa:1}. Therefore we make the following conjecture:
\begin{conjecture}\label{conj:function-field-larger}
Let $k$ be a number field. Then for any transitive permutation group $G$, we have
\begin{equation}
	b_M(G, k) \le b(G,k) \le b_T(G, k).
\end{equation}
\end{conjecture}
It is worth mentioning that at this moment, the issue from the embedding problem in T\"urkelli's modification seems to only exist for number fields. We are not aware of any counterexamples over function fields, either for Conjecture \ref{conj:T\"urkelli} or Problem \ref{que:embedding-cyclotomic}. Given Theorem \ref{thm:embedding-function-field}, it seems to be plausible to expect we don't have such failures for Problem \ref{que:embedding-cyclotomic} over function field. Though a positive answer in full generality seems to be difficult to prove.  

\subsection{General Invariants}
Malle's conjecture is also interesting because of its connection to other asymptotic questions. In fact, field enumeration naturally occurs when one studies  statistical questions for all $G$-extensions as a family. For example, it is exactly the denominator appearing in Cohen-Lenstra type heuristics. Choosing the ordering for field enumeration is important from this perspective. It has been noticed that ordering fields by discriminants does not always produce the predicted average number from Cohen-Martinet heuristics , e.g. $G = \Z/4\Z$ \cite{bartel2020class}. In fact, this phenomenon already exists for $G = D_4$ in \cite{CDyDO02} and Kl\"uners' examples of wreath products \cite{Klu}. To avoid such problems, people have worked in the past with conductor or product of ramified primes \cite{Woo10a,D4VS}. It is suggested first in \cite{Woo10a} to use the product of ramified primes as the counting invariant and it is conjectured by \cite{bartel2020class} that it is always good for the purpose of Cohen-Lenstra heuristics, and is actually used for proving versions of Cohen-Lenstra heuristics over function fields for general Galois groups. With notation introduced in Section \ref{sec:b}, it naturally extends Conjecture \ref{conj:Malle}, Conjecture \ref{conj:Malle-function-field} and Conjecture \ref{conj:T\"urkelli} (with all counterexamples carried over):
\begin{conjecture}[Generalized Malle's Conjecture]\label{conj:Malle-general-invariant}
Given a global function field $Q$ and a transitive permutation group $G\subset S_n$ with $(|G|, \charf(Q)) = 1$. Let $\text{inv}$ be a counting invariant(see Definition \ref{def:invariants}). Then there exist positive constants $a(G^{\text{inv}}) \in \Z$ and $b(G^{\text{inv}},Q) \in \Z$ such that
	$$N_Q(G^{\text{inv}}, X) = \Theta( X^{1/a(G^{\text{inv}})} \ln^{b(G^{\text{inv}},Q)-1} X),$$
	where $b(G^{\text{inv}}, Q) = b_M(G^{\text{inv}}, Q)$ (replace $\Theta$ with $\sim$ when $Q$ is number field). Similarly $b_T(G^{\text{inv}})$ given in Section \ref{sec:b} is the analogue of T\"urkelli's modified constant.
\end{conjecture}

The reason for the aforementioned issue is very often concluded to be the following group theoretic reason. In the upcoming paper of the author with Alberts, Lemke-Oliver and Wood, we define \textit{concentrated} groups, to capture this feature:
\begin{definition}[Concentrated Group, \cite{ALOWW}]\label{def:concentrated-group}
	We say a transitive permutation group $G\subset S_n$ is \emph{concentrated}, or \emph{concentrated in $N$}, if there exists a proper normal subgroup $N$ of $G$ such that 
	\begin{equation}
			\langle s\mid \ind(s) = \min_{g\neq e \in G} \ind(g) \rangle \subset N,
	\end{equation}
	(see definition of $\ind(\cdot)$ in Definition \ref{def:ind}). More generally, given an counting invariant $\text{inv}$ (see Definition \ref{def:invariants}), we say \emph{$G$ is concentrated in $N$ with respect to $\text{inv}$} if $\langle s\mid \exp(s) = \min_{g\neq e \in G} \exp(g) \rangle \subset N$.
\end{definition}
\noindent If $G$ is concentrated in $N$ with respect to the discriminant or some other invariants, then the number of $G$-extensions with a fixed $G/N$-quotient is expected to have positive density among all $G$-extensions. In fact, this is exactly why the method in Alberts-Lemke-Oliver-Wang-Wood \cite{ALOWW} is so effective in proving many more cases of Malle's conjecture when $G$ is concentrated. We term this phenomena on the field counting side by
\begin{definition}[Big Fiber]\label{def:big-fiber}
	Let $\mathcal{F}(G^{\text{inv}}, X)$ be the number of $G$-extensions $K$ with $\text{inv}(K)\le X$. We say a field counting question $G^{\text{inv}}$ over base field $Q$ has a \emph{big fiber over $M$} if there exists a nontrivial field extension $M\neq Q$ such that 
	\begin{equation}
		\liminf_{X\to \infty} \frac{\sharp\{  K\in \mathcal{F}(G^{\text{inv}}, X)\mid M\subset \tilde{K}\}}{\sharp\{  K\in \mathcal{F}(G^{\text{inv}}, X)\}} >0.
	\end{equation}
\end{definition}

It is very tempting to imagine that $G$ being concentrated is equivalent to the existence of big fiber(s) in counting $G$-extensions (with respect to any invariant), or even stronger, that $G$ being non-concentrated is equivalent to guarantee a good constant like an Euler product, which suggests good independence among different $p$. Unfortunately, both hopes are not true. In fact, even for $G = S_3$, the field counting question $S_3^{\text{rad}}$ with product of ramified primes is already proven by \cite{shankar2024asymptotics} to have one big fiber over $\Q(\mu_3)$, even though $S_3$ is not concentrated with respect to the product of ramified primes, not to mention the constant being one Euler product! 

It is exactly for the same reason for Kl\"uners to get his counter examples and for $S_3^{\text{rad}}$ to have one big fiber. It is then very natural to push this reasoning further to get counterexamples for Conjecture \ref{conj:Malle-general-invariant} for general non-concentrated invariants. 

Denoting $b_M(G^{\text{rad}}, Q)$ to be the $b$-constant appearing in Conjecture \ref{conj:Malle-general-invariant} for counting $G$-extensions by the product of ramified primes, we show that actually the original counterexamples of Kl\"uners are already counterexamples for Conjecture \ref{conj:Malle-general-invariant}.
\begin{theorem}\label{thm:example-rad-1}
	Let $\ell$ be an odd prime and $G = C_{\ell} \wr C_{\ell-1}$. When $\ell\ge 5$, we have
	$$b(G^{\text{rad}}, \Q) > b_M(G^{\text{rad}}, \Q).$$
\end{theorem}
\noindent In fact, the lower bound is exactly established from counting $G$-extensions containing the $\Q(\mu_{\ell})$. We also show the same phenomena to hold for other groups $G= C_{\ell} \wr C_m$ with $m| \ell-1$ among Kl\"uners' counterexamples, see Lemma \ref{lem:Cl-wr-Cl-1}.

We make a couple comments on these examples. Firstly, though it seems surprisingly simple, somehow it has been ignored by experts, see \cite{koymans2023malle}. Secondly as it can be seen, such a problem does not require subtle constraints on the underlying group structure of $G$. To indicate this in contrast to nilpotent examples in \cite{koymans2023malle}, we also give a convenient infinite family of nilpotent examples.
\begin{theorem}\label{thm:example-rad-2}
	Let $\ell$ be an odd prime and $G = C_{\ell^2} \wr C_{\ell}$. When $\ell\ge 3$, we have
	\begin{equation}
		b(G^{\text{rad}}, \Q) > b_M(G^{\text{rad}}, \Q).
	\end{equation}
\end{theorem}

Thirdly, Theorem \ref{thm:example-rad-1} and Theorem \ref{thm:example-rad-2} only takes advantage of the fact that $b_M(G^{\text{rad}}, \Q) < b_T(G^{\text{rad}}, \Q)$ (as well as examples in \cite{koymans2023malle}). All issues we point out in Theorem \ref{thm:example-1} persist for general invariants, as expected. That means for $\text{rad}$, the same issue appears for $b_T(G^{\text{rad}}, Q)$ as well! See Example \ref{exm:C3wrC4-bt} and \ref{exa:koymans-pagano}. 

Fourthly, this issue of $b_M(G^{\text{rad}}, Q)$ demonstrated by Theorem \ref{thm:example-rad-1} and \ref{thm:example-rad-2}, does not conflict with previous results in Cohen-Lenstra heuristics on function field side, e.g. \cite{liu2019predicted,sawin2023conjectures,liu2022non}, where the product of ramified primes is being used, since all statistics are proved for $G$-extensions that do not contain constant extensions. Finally these examples show that there doesn't seem to exist a unifying invariant, discriminant or product of ramified primes, that solves all the troubles once and for all. 

\subsection{A New Proposal}
Now in retrospect of both types of examples presented, the intermediate subfield containing/not containing roots of unity can affect $b$ constants easily from different mechanisms. We don't try to design any invariant to suppress the influence from roots of unity. The solution we propose in the following is a refined version of Malle's conjecture with general invariants following the exact spirit from global function fields, where $G$-extensions with particular fixed constant extensions are naturally grouped together. Here we split up $N_{Q}(G^{\text{inv}}, X)$ into a finite set of subquestions according to the intersection of $G$-extensions with relevant cyclotomic extensions:

\begin{conjecture}[Refined Malle's Conjecture]\label{conj:refined-malle}
	Let $G\subset S_n$ be a finite group, $Q$ a global field with $(|G|, \charf(Q)) = 1$. Let $d = \text{lcm}_{\exp(g) = \exp(G)} \ord(g)$. For any cyclotomic subfield $F\subset Q(\mu_{d})$ with $\Gal(F/k) = B$, i.e. a surjectiion $\phi:G_Q \to B$ that factors through $\Gal(Q(\mu_d)/Q) \to B$, and a surjection $\pi: G \to B$, we define $N_{Q, \pi, \phi}(G^{\text{inv}}, X)$ to be the number of continuous surjective liftings $\tilde{\phi}$ such that: 
	\begin{itemize}
		\item it makes the diagram commute:
	\begin{center}
    \begin{equation}
		\begin{tikzcd}	
			&  &  & G_Q \arrow[ld, "\tilde{\phi}" ', dashed]\arrow[d,"\phi",two heads] &  \\		
			0\arrow{r} & \Ker(\pi) \arrow{r} & G \arrow[r, "\pi"] & B \arrow{r} & 0 \\
		\end{tikzcd}
            \label{diag:embedding-2}
    \end{equation}
	\end{center}
	\item the fixed Galois field $K(\tilde{\phi})$ associated to $\tilde{\phi}$ satisfy $K({\tilde{\phi}})\cap Q(\mu_d) = F$
    \item $\text{inv}(K(\tilde{\phi})) \le X$
    \end{itemize}
    We conjecture that either the above embedding problem  is not solvable with surjections and $N_{Q, \pi, \phi}(G^{\text{inv}}, X)  = 0$ or
	\begin{equation}
		N_{Q, \pi, \phi}(G^{\text{inv}}, X) = \Theta( X^{1/a(\pi)} \ln^{b(\pi, \phi)-1} X),
	\end{equation}
	(when $Q$ is a number field, replace $\Theta$ with $\sim$) where
	\begin{equation}
		a(\pi)= \min_{g\neq e} \{  \exp(g) \mid g\in \Ker(\pi)  \}, \quad \quad b(\pi, \phi)= |\mathcal{S}_{min}(\Ker(\pi)^{\text{inv}})/G(\pi, \phi)|
	\end{equation}
	where $\mathcal{S}_{min}(\Ker(\pi)^{\text{inv}}):=\{ g\in \Ker(\pi): \exp(g) =a(\pi)\}$ and $G(\pi, \phi):= G\times_{\pi, \phi} \Gal(Q(\mu_d)/Q) = \{(x,y) \mid \pi(x) = \phi(y) \}$ is naturally considered as a subgroup of $G\times (\Z/d\Z)^{\times}$. The action of $(x,y)\in G(\pi, \phi)$ on $\mathcal{S}_{min}(\Ker(\pi)^{\text{inv}})$ is $(x,y)\cdot g = x^{-1} g^y x$. 
\end{conjecture}

Very recent progresses on function field side \cite{landesman2025homological} over rational global function fields supports both Conjecture \ref{conj:refined-malle} and Conjecture \ref{conj:new-malle}, in the sense that when no lifting problem occurs, there does not exist new issues for $b$ over function fields. Analogue of $N_{Q, \pi, \phi}(G,X)$ is also discussed in \cite[Conjecture 5.1]{koymans2023malle} for nilpotent $G^{\text{rad}}$ without noticing the issue of lifting property, see Example \ref{exa:koymans-pagano} where we show lifting problems do occur for $p$-groups over number fields as well. 

In terms of the original Malle's Conjecture, we also give the following consequence on $N_{k}(G,X)$ over number fields. We see that
\begin{equation}
	N_k(G,X) = \sum_{\pi, \phi} N_{Q, \pi, \phi}(G,X)
\end{equation}
where $(\pi, \phi)$ ranges over the tuples in Conjecture \ref{conj:refined-malle}. Notice this is a finite disjoint sum. It then follows as a consequence of Conjecture \ref{conj:refined-malle}, that
\begin{conjecture}[New Malle's Conjecture]\label{conj:new-malle}
	Given a transitive permutation group $G\subset S_n$ and a number field $k$. Denote $d = \text{lcm}_{g, \ind(g) = \ind(G)} \ord(g)$. We conjecture that
	\begin{equation}
		b(G, k) = \max \quad  b(\pi, \phi)
	\end{equation}
where the maximum ranges over the finitely many pair of continuous surjective $(\pi, \phi)$ such that: \\
1) $\phi:G_k \to G/\Ker(\pi)$ exactly cut out cyclotomic subfields $k(\mu_d)/k$ and $\ind(\Ker(\pi)) = \ind(G)$; \\
2) $\phi$ can be lifted to $\tilde{\phi}$ in Diagram (\ref{diag:embedding-2}).
\end{conjecture}
\noindent Comparing with Theorem \ref{thm:bt-simple}, we now have one more condition on the lifting property of $\phi$. 

We believe that for each $(\pi, \phi)$, the counting function $N_{Q,\pi, \phi}(G,X)$ should stand as each individual problem, as this is exactly the analogue of T\"urkelli's idea when brought over to number fields. We also expect many interesting statistical questions for each of these families, as the existence of roots of unity or interserction of cyclotomic fields not only affect field distribution, but also various arithmetic objects. 

Finally, we give the organization of this paper. In Section \ref{sec:b}, we discuss predictions of $b(G,Q)$ by Malle and T\"urkelli. In particular, we give a simplified definition of $b_T(G, Q)$ in Theorem \ref{thm:bt-simple}. In Section \ref{sec:counting}, we verify $b(G,Q)$ in all examples we listed in Theorem \ref{thm:example-1}, Theorem \ref{thm:example-rad-1} and Theorem \ref{thm:example-rad-2}. This includes computing the predictions from the group theoretic side, and also carrying over inductive argument to prove the true $b$. We also discuss a couple differences for counting function field and number fields in this section. In Section \ref{sec:embed}, we discuss the difference of Problem \ref{que:embedding-cyclotomic} over function fields and number fields, and give an explicit criteria when $G$ is solvable and $G$ is abelian respectively. 

\section*{Notation}
We use $Q$ to denote the base field when we discuss results for general global fields, and use $k$ to denote the base field when we discuss results for number fields. Without further explanation, $G$ is always a transitive permutation group. We use $G^{\text{inv}}$ to denote counting problem for $G$-extensions by an invariant $\text{inv}$. Frequent examples of invariants include $\text{rad}$ which is product of ramified primes and $\text{disc}$ which is discriminant. We use $\exp:G \to \Z_{>0}$ to denote the exponent function for a general $\text{inv}$. For $\text{disc}$, this function is historically denoted by $\text{ind}$. We denote the counting function for $G$-extensions over $Q$ with respect to $\text{inv}$ by $N_{Q}(G^{\text{inv}},X)$. There are notations for various $b$-constants. When $b$ is decorated with subscript $b_M$( respectively $b_T$), we mean the conjectural value by Malle (respectively T\"urkelli) or by Malle's (respectively T\"urkelli's) principle. Therefore $b_{*}(G^{\text{inv}}, Q)$ means Malle's/T\"urkelli's/true value of the power of log in counting $G$-extensions over base field $Q$ with respect to invariant $\text{inv}$. We use $\pi: G \to G/N \simeq B$ to denote the projection from $G$ and always use $N$ to denote $\Ker(\pi)$. We use $\phi: G_Q \to B$ to denote a general abelian extension. The counting function $N_{Q,\pi, \phi}(G^{\text{inv}}, X)$ counts the number of $G$-extensions with respect to the conditions listed in Conjecture \ref{conj:refined-malle}. The letter $d$ is the lcm of order of elements in $G$ with minimal $\exp(g)$. $G(\pi, \phi)$ is a subgroup of $G\times (\Z/d\Z)^{\times}$, see its definition in (\ref{eqn:Gpiphi}). The $\phi$-twisted action is defined in Definition \ref{def:phi}. We stick to $b(\pi, \phi)$ to denote the orbits of $|\mathcal{S}_{min}(\Ker(\pi)^{\text{inv}})/G(\pi, \phi)|$ under the $\phi$-twisted action from $G(\pi, \phi)$. We listed many conjectures in the introduction as they have been studied or proposed in the past: Conjecture \ref{conj:Malle}, \ref{conj:Malle-function-field}, \ref{conj:T\"urkelli}, \ref{conj:Malle-general-invariant} all have counterexamples. We do not know of counterexamples for Conjecture \ref{conj:function-field-larger} and Conjecture \ref{conj:refined-malle}. We also do not know counterexamples for Conjecture \ref{conj:T\"urkelli} for global function fields.

\section{Description of $b(G,k)$}\label{sec:b}
In this section, we give a precise description of Malle's prediction $b_M(G,k)$ and T\"urkelli's modifiction $b_T(G,k)$. Although the original conjecture of Malle and T\"urkelli are both made only for discriminant, for efficiency in discussing all theorems together, we define them with respect to general invariants once for all. 

Following the spirit in \cite[Section 2.1]{Woo10a}, we give the following definition for a general counting invariant. 
\begin{definition}[Counting Invariant]\label{def:invariants}
	Let $G\subset S_n$ be a finite permutation group and $\mathcal{C}(G)$ be the set of conjugacy classes of $G$. Let $Q$ be a global field. Let $\exp: \mathcal{C}(G) \backslash \{ e \} \to \Z_{> 0}$ be a function where $\exp(g) = \exp(g^k)$ for any $k$ that is relatively prime to $\ord(g)$. For each place $v||G|\cdot \infty$ (above $|G|$ and $\infty$ of $\Q$), we choose $\exp_v: \Sigma_v \to \R_{> 0}$ where $\Sigma_v$ is the set of continuous group homomorphisms $\rho_v: G_{Q_v} \to G\subset S_n$ (up to conjugation in $S_n$). Then for each $G$-extension $K/Q$ with $\Gal(K/Q) \simeq G\subset S_n$, equivalently given by a continuous surjective group homomorphism $\rho_K: G_Q \to G$, we define the \emph{counting invariant} for $K$ associated to $f$ and $f_v$ to be an integer denoted by $\text{inv}(K)$:
	\begin{equation}
			\text{inv}(K) = \prod_{v||G|} |v|^{\exp_v(\rho_v)} \prod_{v|\infty} \exp_v(\rho_v) \prod_{v\nmid |G|} |v|^{\exp(y_v)},
	\end{equation}
	where $y_v$ is any tame inertia generator at $v$ in $\Gal(K/Q)$.
\end{definition}
For a general invariant $\text{inv}$, we use $G^{\text{inv}}$ to denote the counting question with $\text{inv}$ and $N_k(G^{\text{inv}}, X)$ to denote the counting function. For example, $G^{\text{rad}}$ denotes the counting question with radical of discriminant. When we do not specify the invariant, our counting invariant for $G\subset S_n$ is the usual discriminant. 

\subsection{Malle's constant $b_M(G,k)$}
We first describe the $a$-constant $a(G)$ in Conjecture \ref{conj:Malle}.
\begin{definition}[Index]\label{def:ind}
Given a transitive permutation group $G\subset S_n$, we define the \emph{index} for $g\in G\subset S_n$ to be 
	$$\ind(g): = n-\sharp\{ \text{cycles of }g \}.$$
	The \emph{index} of $G$ is defined as
	$$\ind(G): = \min_{g\neq e} \ind(g).$$
\end{definition}
\noindent The integer $a(G)$ is exactly $\ind(G)$. Notice that $\ind(\cdot)$ is exactly the function $\exp(\cdot)$, i.e. the power of $p$ in the discrimint when the tame inertia is $g$, when the counting invariant $\text{inv}$ is discriminant, we in general define
\begin{equation}
	\exp(G^{\text{inv}}):= \min_{g\neq e} \exp(g).
\end{equation}
Recall that
\begin{definition}[Cyclotomic Action]
	Given any field $Q$, the \emph{cyclotomic character} is the canonical homomorphism
	$$\chi_{cyc}: G_Q \to \Aut(\mu_{\infty})\subset \hat{\Z}^{\times} = \lim_{\substack{\longleftarrow\\n}}(\Z/n\Z)^{\times}.$$
	We define the \emph{cyclotomic action} of $G_Q$ on a finite group $G$ as $\sigma(g) = g^{\chi_{cyc}(\sigma)}$.
\end{definition}
\noindent Notice that since $G$ is finite, $g^{\chi_{cyc}(\sigma)}$ only depends on the image $\chi_{cyc}(\sigma)$ in $\Z/|G|\Z$. More concretely, denote $d = |G|$, it suffices to consider the image of $G_{Q}$ into $(\Z/d\Z)^{\times}$. If $\sigma \in G_{Q}$ maps $\sigma(\mu_d) = \mu_d^a$, then $\sigma(g) = g^a$. In fact if a certain group element $g\in G$ has order $m$, then the action of $\sigma \in G_Q$ on $g$ can be computed as $g^{\chi_{cyc}(\sigma)}$ via its image in $(\Z/m\Z)^{\times}$ already. 

\begin{definition}[Malle's $b$]
Denote $\mathcal{C}_{min}(G^{\text{inv}})$ to be the set of conjugacy classes of $G$ with minimal $\exp$. We now define
\begin{equation}\label{eqn:bM}
	b_M(G^{\text{inv}},Q): = | \mathcal{C}_{min}(G^{\text{inv}})/ G_Q|,
\end{equation}
under the cyclotomic action from $G_Q$. 
\end{definition}
\noindent Malle in \cite{Mal04} conjectures that $b_M(G, Q) = b(G, Q)$ where $G$ stands for the natural discriminant as the counting invariant. 
\begin{remark}
This action factors through $\Gal(Q(\mu_d)/Q)$ where $d$ is $\text{lcm}_{g \in \mathcal{C}_{min}(G^{\text{inv}}) } \ord(g)$. This means that the base field dependence in $b_M(G^{\text{inv}},Q)$ only comes from $\Gal(Q(\mu_d)/Q)\subset (\Z/d\Z)^{\times}$. 
\end{remark}
\noindent We demonstrate the computation of $b(G,k)$ in the following examples. It is also stated in \cite{Klu12}.
\begin{example}[Wreath Product]\label{Example:wreath-product}
	Let $G = T\wr B\subset S_{mn}$ be the wreath product as a permutation group, where $T\subset S_m$ and $B\subset S_n$. Since $\ind( (t_1, \cdots, t_m) \rtimes \sigma ) \ge \ind( (t_1, \cdots, t_m)\rtimes e)$, we see $\ind(G) = \ind( (t, e, \cdots, e) \rtimes e)$ where $\ind(t) = \ind(T)$. It is now easy to see that the only elements with this index are exactly in such a form. Therefore $\ind(G) = \ind(T)$ and $\mathcal{C}_{min}(G)$ is in bijection with $\mathcal{C}_{min}(T)$. Moreover the action from $G_Q$ is identical. Thus $b_M(G, Q) = b_M(T,Q)$. 
\end{example}
\begin{remark}\label{rmk:malle-inconsist}
Given the fact that the Galois group of a relative $T$-extension over a $B$-extension has Galois group embedded into $T\wr B$, this example immediately implies that Conjecture \ref{conj:Malle} cannot be true in general, since it is not consistent with itself. Indeed for some $T$ and $B$ (e.g. $T= C_3$ and $B=C_2$), the total prediction of all permutation Galois groups that arise as a relative $T$-extension over a $B$-extension add up to be smaller than the number of relative $T$-extensions for a fixed cyclotomic $B$-extensions. 
\end{remark}

\subsection{T\"urkelli's constant $b_T(G,k)$}\label{sec:tb}
In this section, we give T\"urkelli's modification on Malle's constant $b$ in \cite{turkelli2015connected}. This is of crucial importance, since it is often not interpreted correctly. We adopt a slightly different language from T\"urkelli, as we believe the current approach has more clarity and leads to simplification. We also derive a simple form of T\"urkelli's Conjecture in Theorem \ref{thm:bt-simple}.

Let $G$ be a transitive permutation group and $N\trianglelefteq G$ a normal subgroup and $Q$ a global field and $\phi: G_Q \to G/N$ be any continuous group homomorphism. Let $\text{inv}$ be the counting invariant and $\mathcal{S}_{min}(N^{\text{inv}})$ be the set of group elements in $N\trianglelefteq G$ with minimal $\exp$. Let $d$ be the lcm of $\ord(g)$ for $g\in \mathcal{S}_{min}(G^{\text{inv}})$. Denote $Q(\phi)$ be the field corresponding to $\ker(\phi)$ and $M(\phi):=Q(\phi)\cap Q(\mu_d)$ is the intersection of $Q(\phi)$ with the relevant cyclotomic fields. We define
\begin{equation}\label{eqn:Gpiphi}
    G(\pi, \phi):= G\times_{\pi, \phi} \Gal(Q(\mu_d)/Q) \subset G\times (\Z/d\Z)^{\times}
\end{equation}
to be a fibered product of $G$ and $\Gal(Q(\mu_d)/Q)$ where the fibering map is the induced maps $\phi:\Gal(Q(\mu_d)/Q) \to \Gal(M(\phi)/Q)$ and $\pi:G \to \Gal(Q(\phi))\to \Gal(M(\phi)/Q)$. 
\begin{definition}[$\phi$-twisted cyclotomic action]\label{def:phi}
Given $G$, $\pi$, $\phi$ and $\text{inv}$. We define the \emph{$\phi$-twisted cyclotomic action} of $G(\pi, \phi)$ on $\mathcal{S}_{min}(\Ker(\pi)^{\text{inv}})$ via:
\begin{equation}
    (x,y)\cdot g = x^{-1} \cdot g^y\cdot x.
\end{equation}
We then define
\begin{equation}
    b(\pi, \phi):= |\mathcal{S}_{min}(\Ker(\pi)^{\text{inv}})/G(\pi, \phi)|.
\end{equation}

\end{definition}
\begin{definition}[T\"urkelli's Modified $b$]\label{def:tb}
Given a global field $Q$ and $G^{\text{inv}}$, 
\begin{equation}
		b_T(G^{\text{inv}}, Q):= \max_{\pi, \phi} \quad  b(\pi, \phi),
\end{equation}
where $\pi$ ranges over all $G\to G/N$ where $a(N) =a(G)$ and $\phi$ ranges among $\phi:G_Q \to G/N$ that factors through $\Gal(Q(\mu_{\infty})/Q)$. 
\end{definition}
\begin{remark}\label{rmk:we-generalize}
T\"urkelli \cite{turkelli2015connected} only formulates his conjecture for the discriminant. We take the liberty of calling it $b_T$ for general invariants, with the same spirit carried over. For $\text{inv} = \text{disc}$, his conjecture states
 	\begin{equation}
 	b_T(G, Q):= \max_{N\triangleleft G, \ind(N) = \ind(G), G/N \text{ is abelian}} b( N, G, Q),
 \end{equation}
where his notation $b(N,G,Q)$ is the maximal $b(\pi, \phi)$, denoted to be $b_{\phi}(N,G,Q)$ by T\"urkelli, among all tuple $(\pi, \phi)$ where $\pi: G\to G/N$ is fixed and $\phi$ varies over all $\phi:G_Q \to G/N$ that factors through $\Gal(Q(\mu_{\infty})/Q) \to G/N$. In particular when we restrict to rational number field $\Q$, $\phi$ varies over all abelian $G/N$-extensions of $\Q$. 
\end{remark}
\begin{remark}\label{rmk:we-simplify}
Our $\phi$-twisted action is defined in a slightly different way in \cite{turkelli2015connected}. In particular, T\"urkelli defines $b(\pi, \phi)$ (in his notation $b_{\phi}(G, N, Q)$) to be the number of orbits of $G_{Q}$ acts on the set of conjugacy classes $C$ of $N$ (conjugation taken only inside $N$ instead of $G$) with minimal $\exp$, i.e. $\{ \text{conjugacy class } c\subset N  \mid \exp(c) = \exp(N)\} $, via the action: for $\sigma \in G_Q$
$$\sigma(c) = \bar{\sigma}^{-1} \cdot  {c}^{\chi_{cyc}(\sigma)} \cdot \bar{\sigma},$$
where $\bar{\sigma} \in G$ can be any preimage of $\phi(\sigma)\in G/N$. Notice that this action of $G_Q$ factors through $\Gal(M(\phi)/Q)$. Notice that 
$$\mathcal{S}_{min}(\Ker(\pi)^{\text{inv}})/ (N\times e\subset G(\pi, \phi)) = \{ \text{conjugacy class } c\subset N \mid \exp(c) = \exp(N)\},$$
it is clear that the number of orbits are the same for the two definitions. 
\end{remark}

\begin{remark}
	In \cite{turkelli2015connected} this definition of $b$-constants is only stated for $\phi$ corresponding to abelian extensions. In \cite{alberts2022statistics} this twisted action and orbits $b(\pi, \phi)$ is defined for general $\phi$. In \cite{alberts2022statistics}, $b(\pi, \phi)$ (under different notation) is conjectured to describe the asymptotic distribution for general $N$-torsors over a fixed $\phi$. This is how we should think about it morally. It is justified with a heuristic argument similarly like how Malle's original conjecture is justified. It is also verified to be the correct $b$ when $N$ is abelian, parallel to that Malle's original conjecture also holds for all abelian extensions. However Kl\"uners' counterexample still stands as counterexamples for these more general twisted conjectures.
\end{remark}

\begin{remark}
The order of conjugation in Definition \ref{def:phi} cannot be reversed. The author thanks Loughran and Santens for pointing this out. Similar arguments are already carried over in \cite{alberts2022statistics}. For the convenience of the reader, we give an explanation. Mainly, this is due to the fact that the tame relator is $xyx^{-1}=y^p$, not the other way around. To derive the heuristics $b$ for counting $G$-extensions $\rho:G_Q \to G$ that $\pi\circ\rho = \phi$ where $\pi: G\to G/N$ and $\phi:G_Q \to G/N$ is given, 
an Euler product can be formed to approximate the generating series 
\begin{equation}
    \prod_{v\in Q} \left(\dfrac{1}{|N|} \sum_{\rho_v: G_{Q_v} \to G, \pi\circ\rho_v=\phi_v} \dfrac{1}{|\text{inv}(\rho_v)|^s} \right).
\end{equation}
When $v\nmid |G|$, $\rho_v$ factor through the tame quotient of $G_{Q_v}:=\langle x_v,y_v\mid x_vy_vx_v^{-1}=y_v^{|v|}\rangle$ and $\text{inv}(\rho_v)= |v|^{\text{exp}(\rho_v(y))}$. Denoting $\delta$ to be the Chebotarev density of primes of chosen behaviors. The primes with Frobenius $(\bar{x},a) \in G/N\times \Gal(\Q(\mu_d)/Q)$, all altogether has a contribution 
\begin{equation}
\begin{aligned}
\delta(\bar{x},a) \cdot \dfrac{|\{(x,y)\in G^2\mid xyx^{-1} = y^{|v|}, \pi(x)=\bar{x},\pi(y)=e, \exp(y) = \exp(N) \}|}{|N|} 
\end{aligned}
\end{equation}
to the order of poles. Therefore the total order is
\begin{equation}
\dfrac{1}{|N|\cdot |G/N\times_{\phi}\Gal(Q(\mu_d)/Q)|}\sum_{ (x,a) \in G(\pi, \phi) } |\{y\in \mathcal{S}_{min}(N^{\text{inv}}) \mid xyx^{-1} = y^{a}\}|
\end{equation}
which is the summation of fixed points of $\phi$-twisted action from $G(\pi, \phi)$ on $\mathcal{S}_{min}(N^{\text{inv}})$, where $xyx^{-1} = y^a$ means $y = x^{-1}\cdot y^a\cdot x$ is fixed by the twisted action. From Burnside's Lemma, we can see this is exactly computing the number of orbits of $G(\pi, \phi)$ on $\mathcal{S}_{min}(\Ker{\pi}^{\text{inv}})$. We can also define the twisted action as $(x,a): y \to x y^{-a} x^{-1}$. It can be easily verified that the number of orbits are the same. Notice this does not make a difference on the original Malle's Conjecture since $(x,a)$ and $(x^{-1},a)$ are both contained in $G\times \Gal(Q(\mu_d)/Q)$ when the fibered product is just the direct product. 
\end{remark}

We make two comments. Firstly, it is clear that $b_M(G^{\text{inv}}, Q) = b(\pi, \phi)$ when $\pi$ and $\phi$ are both trivial, therefore
\begin{lemma}
Given $G\subset S_n$ and a global field $Q$, it always holds that
	\begin{equation}
		b_M(G, Q) \le b_T(G,Q).
	\end{equation}
This also holds for general counting invariants $\text{inv}$.
\end{lemma}

Secondly, over function fields $Q^{cyc}$ is simply $Q \cdot \bar{\mathbb{F}}_p$ which is cyclic, however over number fields $Q^{cyc}$ is not, in fact it is countably generated over any number field. This means that determining $b_T(G, Q)$ involves checking for infinite many $\phi$ for number fields but finitely many for function fields. Therefore we make the following simplification for number fields. 

We first compare $b(\pi, \phi)$ among different $\phi$.
\begin{lemma}\label{lem:bphi-comparison}
    Given $G$, $\pi$, $\phi$, $Q(\phi)$, $M(\phi)$ and $\text{inv}$. Let $\phi':G_Q \to \Gal(M(\phi)/Q)$ induced by $\phi$. Denote the $B_0:=\Gal(Q(\phi)/M(\phi))\subset B$ and let $\pi':G \to G/\pi^{-1}(B_0)$. Then
    \begin{equation}
        b(\pi, \phi) \le b(\pi', \phi').
    \end{equation}
\end{lemma}
\begin{proof}
 It follows from $\mathcal{S}_{min}(\Ker(\pi')^{\text{inv}})\subset \mathcal{S}_{min}(\Ker(\pi')^{\text{inv}})$ and $G(\pi, \phi) = G(\pi', \phi')$.
\end{proof}

Applying Lemma \ref{lem:bphi-comparison}, we immediately get the following simplification to compute T\"urkelli's constant:
\begin{theorem}\label{thm:bt-simple}
Given $G\subset S_n$, $Q$ a global field, and $d = \text{lcm} \{ \ord(g) \mid \exp(g) = \exp(G) \}$. We have the following equality:
\begin{equation}
	b_{T}(G^{\text{inv}}, Q) =\max_{\pi, \phi} \quad b(\pi, \phi),
\end{equation}
where the maximum is only taken among all $\pi$ with $\exp(\Ker(\pi)) = \exp(G)$ and $\phi:G_Q\to G/N$ that factors through $\Gal(Q(\mu_d)/Q) \to G/N$.
\end{theorem}

\noindent Notice that given $Q(\mu_d)/Q$ being finite, it is now a finite checking with $\phi$ exactly corresponding to subfields of $Q(\mu_d)$. 
\begin{remark}
If $\mathcal{S}_{min}(G^{\text{inv}})\neq G$, then for $\pi$ with $\Ker(\pi)\supset \mathcal{S}_{min}(G^{\text{inv}})$, it suffices to look at $\phi$ that are maximal (in the sense that $G(\pi, \phi)$ is minimal). It may be non-unique. 
\end{remark}

In general, this is the best one can say: there is no further comparison for $b(\pi, \phi)$ in Theorem \ref{thm:bt-simple} among different cyclotomic $\phi$. See Example \ref{exm:C3wrC4-bt}. Lemma \ref{lem:bphi-comparison} also shows that there is no inconsistency in Conjecture \ref{conj:T\"urkelli} and \ref{conj:refined-malle} in the sense of Remark \ref{rmk:malle-inconsist}. It also motivates Conjecture \ref{conj:refined-malle}.

\section{Computation of $b(G,Q)$}\label{sec:counting}
In this section, we verify various $b$-constants in Conjecture \ref{conj:Malle}, Conjecture \ref{conj:T\"urkelli}, Conjecture \ref{conj:Malle-general-invariant} and the true value.
\subsection{Group Constant Computation}
In this section, we give the group theoretic computation of predictions for $b$-constants from Malle and T\"urkelli. We first give the computation towards showing Theorem \ref{thm:example-1}.
\begin{lemma}
	Let $\ell$ be an odd prime number and $d = \prod_i p_i^{r_i}$ where $p_i$ are all prime numbers. Let $G = C_{\ell} \wr C_{d}\subset S_{\ell d}$ and $\gcd(q, |G|) =1$. If $\gcd(d, \ell-1) = \prod_i p_i^{s_i}$ then
	$$b_T(G, \Q) =\gcd(d, \ell-1) \quad \quad b_M(G, \Q) = 1.$$
\end{lemma}
\begin{proof}
	By Example \ref{Example:wreath-product}, we have $b_M(G, \Q) = b_M(C_{\ell}, \Q) = 1$.  We now compute $b_T(G, \Q)$. Firstly, we only need to consider those $N\supset C_{\ell}^{d}$, since $C_{\ell}^d$ is the normal closure of any minimal index element. By Theorem \ref{thm:bt-simple} and Lemma \ref{lem:bphi-comparison}, it suffices to consider $\phi: G_{\Q} \to \Z/ \gcd(d, \ell-1) \Z$ correspond to the unique subfield of $\Q(\mu_{\ell})$ with degree $\gcd(d, \ell-1) = \prod_i p_i^{s_i}$. We now compute $b(\pi, \phi)$ where $\pi$ is the unique surjection from $G$ to $\Z/\prod_i p_i^{s_i} \Z$. We use T\"urkelli's original definition, Remark \ref{rmk:we-simplify}, in computing $b(\pi, \phi)$. It is easy to count that the number of conjugacy classes in $N$ with minimal index is $|\mathcal{C}_{min}(N)| = (\ell-1)\gcd(d, \ell-1)$. The action from $G_{\Q}$ factors through $\Z/\prod_i p_i^{s_i} \Z$, therefore it is enough to consider the action from the generator of $\Z/\prod_i p_i^{s_i} \Z$. The orbit length for each class is exactly $\ell-1$, therefore the number of orbits is exactly $\gcd(d, \ell-1) = \prod_i p_i^{s_i}$. 
\end{proof}

 We next give the computation towards showing Theorem \ref{thm:example-rad-1}
\begin{lemma}\label{lem:Cl-wr-Cl-1}
	Let $\ell$ be an odd prime and $m|\ell-1$ with $m>2$, $G = C_{\ell} \wr C_{m}$ and $N = C_{\ell}^{m}$. Let $\phi$ corresponds to the unique $C_m$ subfield contained in $\Q(\mu_{\ell})$. We then have
	\begin{equation}
		b(\pi, \phi) \gg b_M( G^{\text{rad}}, \Q).
	\end{equation}
	For $m = \ell-1$ and $\ell\ge 5$, 
	\begin{equation}
		b(\pi, \phi) > b_M( G^{\text{rad}}, \Q).
	\end{equation}
\end{lemma}
\begin{proof}
The conjugacy classes of $C_{\ell} \wr C_{\ell-1}$ come in two types: contained in $N$ and outside of $N$. Within $N$, the class represented by $(a_1, \cdots, a_{\ell-1}) \rtimes e $ contains all rotations of $a_i$, i.e. , $(a_i, a_{i+1}, \cdots, a_{\ell-1}, a_1, \cdots, a_{i-1}) \rtimes e$ for certain $i$. Outside $N$, we have $(a_1, \cdots, a_{\ell-1}) \rtimes \sigma$ conjugate to $(b_1, \cdots, b_{\ell-1}) \rtimes \tau$ if and only if $\tau = \sigma$ and $\sum_i a_i = \sum _i b_i$. 
	
We first compute $b_M(G^{\text{rad}}, \Q)$. We first count the number of $G_{\Q}$-orbits within $N$. Notice that the conjugation of $G$ on $N$ purely comes from $G/N = C_{\ell-1}$ and all nontrivial elements in $N$ have order $\ell$, it then suffices to count the orbits for $G/N \times \Gal(\Q(\mu_{\ell})/\Q)$ acting on $X = N\backslash \{ e\}$ where $(x,y) \cdot n = x^{-1}\cdot  n^{\chi_{cyc}(y)} \cdot x$. By Burnside's Lemma, the number of orbits is
\begin{equation}
|X/ (G/N \times \Gal(\Q(\mu_{\ell})/\Q)) |	=  \frac{1}{|C_{\ell-1} \times C_{\ell}^{\times}|}\sum_{g\in C_{\ell-1} \times C_{\ell}^{\times}} |X^g|,
\end{equation}
Let $g = (r, s)$, then $X^g$ corresponds to the non-trivial eigenvectors of $r$ with eigenvalue $s$. If $r$ generates $C_{\ell-1}$, then as a linear operator $r$ satisfies $r^{\ell-1}-1 = \prod_{\lambda \in \F_{\ell}^{\times}} (r- \lambda) = 0$, which shows that every scalar is an eigenvalue and each eigenvalue has one dimensional eigenspace. Thus we compute that $|X^g|$ for $(r, s)$ is $\ell-1$ for this case. Similarly, when $r$ has order smaller than $\ell-1$, the operator $r$ has $(\ell-1)/ \ord(r)$ identical invariant spaces with dimension $\ord(r)$, within which all $\ord(r)$-th roots of unity in $\F_{\ell}^{\times}$ are eigenvalues with one dimensional eigenspace. Therefore if $s^{\ord(r)} = 1$, we obtain $|X^g| = \ell^{(\ell-1)/\ord(r)} -1$. Now we compute the number of orbits within $N$ is
\begin{equation}
 \frac{1}{(\ell-1)^2} \sum_{r \in C_{\ell-1}} \ord(r) \cdot \big(\ell^{(\ell-1)/\ord(r)}-1 \big)
		\le  \frac{1}{(\ell-1)^2} \Big( \ell^{\ell-1}-1 + (\ell-2)(\ell-1)\cdot \ell^{(\ell-1)/2} \Big),
\end{equation}
with the leading term comes from $g=e$. Notice the total number of classes outside $N$ is $(\ell-2)\cdot \ell$, we then have
\begin{equation}
	b_M(G^{\text{rad}}, \Q)\le  \frac{1}{(\ell-1)^2} \Big( \ell^{\ell-1}-1\Big) + \frac{\ell-2}{\ell-1}\cdot \ell^{(\ell-1)/2} + (\ell-2)\cdot \ell.
\end{equation}

Now we compute $b(\pi, \phi)$ where $\phi: G_{\Q} \to C_{\ell-1}$ corresponds to $\Q(\mu_{\ell})/\Q$. By Burnside's Lemma, the number of orbits is only the sum over $(r,r)$ in previous computation:
\begin{equation}
b(\pi, \phi) =   \frac{1}{\ell-1} \sum_{r \in C_{\ell-1}} \big(\ell^{(\ell-1)/\ord(r)}-1 \big) \ge \frac{|X^e| + (\ell-2)(\ell-1)}{\ell-1}  =  \frac{\ell^{\ell-1}-1}{\ell-1} + (\ell-2). 
\end{equation}
It is checked that when $\ell$ is large enough (i.e. $\ell \ge 5$) we have
$$b(\pi, \phi)> b_M(G^{\text{rad}}, \Q).$$	

For $m|\ell-1$ and $G = C_{\ell} \wr C_m$, we can similarly compute $b_M(G^{\text{rad}}, \Q)$ and $b(\pi, \phi)$ where $\phi: G_{\Q} \to C_m$ corresponds to the unique $C_m$ quotient of $\Q(\mu_{\ell})$. The number of conjugacy classes outside $N$ is $(m-1)\cdot \ell$. For elements inside $N$, we consider $C_m \times \Gal(\Q(\mu_{\ell})/\Q)$ acting on $X = N\backslash e$ to determine its contribution to $b_M(G^{\text{rad}}, \Q)$. Thus by Burnside's Lemma, we have
\begin{equation}
	\begin{aligned}
			b_M(G^{\text{rad}}, \Q) \le &(m-1)\cdot \ell + \frac{1}{(\ell-1)m} \sum_{r\in C_m} \ord(r)\cdot (\ell^{m/\ord(r)}-1)\\
		 \le &  \frac{\ell^{m}-1}{(\ell-1)m} + \frac{m-1}{\ell-1} \cdot \ell^{m/2} +  (m-1)\cdot \ell.
	\end{aligned}
\end{equation}
To compute $b(\pi, \phi)$, it suffices to consider $\Gal(\Q(\mu_{\ell})/\Q)$ acting on the same set $X$ with the same action that $x\cdot n =x^{-1}\cdot n^{\chi_{cyc}(x)} \cdot x$. The orbit number is at least 
\begin{equation}
	b(\pi, \phi) =  \frac{1}{\ell-1} \sum_{r\in C_{m}} \Big( \ell^{m/\ord(r)} -1  \Big) \ge \frac{\ell^m -1}{\ell-1} + \frac{m-1}{\ell-1} \cdot (\ell-1) = \frac{\ell^m -1}{\ell-1} + (m-1).
\end{equation}
Therefore for $m>2$, we have $b(\pi, \phi)  \gg b_M(G^{\text{rad}}, \Q)$.
\end{proof}

 We finally give the computation towards showing Theorem \ref{thm:example-rad-2}.
\begin{lemma}\label{lem:Cl2-wr-Cl}
	Let $\ell$ be an odd prime, $G = C_{\ell^2} \wr C_{\ell}$ and $N= C_{\ell^2}^{\ell}$. For $\ell\ge 3$, and $\phi$ corresponds to the unique $C_{\ell}$-extension only ramified at $\ell$ over $\Q$, we have
	$$b(\pi, \phi) > b_M(G^{\text{rad}}, \Q). $$
\end{lemma}
\begin{proof}
	Similarly with Lemma \ref{lem:Cl-wr-Cl-1}, there are two types of conjugacy classes of $C_{\ell^2}\wr C_{\ell}$. The description of these classes also follow exactly the same rule. We focus on those contained in $N$, since the number of conjugacy classes outside $N$ is bounded from above by $(\ell-1)\ell^2$.
	
	We first compute $b_M(G^{\text{rad}}, \Q)$. The cyclotomic action from $G_{\Q}$ on conjugacy classes within $N$ is from $\Gal(\Q(\mu_{\ell^2})/\Q)$. We apply Burnside's Lemma where $X = C_{\ell^2}^{\ell} \backslash \{e \}$ acted by $C_{\ell} \times \Gal(\Q(\mu_{\ell^2})/\Q)$. The action from $C_{\ell}$ is rotation and the action from $\Gal(\Q(\mu_{\ell^2})/\Q) \simeq \C_{\ell^2}^{\times}$ is scalar multiplication. We have $|X^e| = \ell^{2\ell} -1$. If $r = 0$ then for $g =(r, s)$ we have $|X^g| = \ell^{\ell} -1$ if $s \equiv 1 \mod \ell$ and $|X^g| = 0$ if not. If $r\neq 0$ generates $C_{\ell}$ and $s = 1$, then $X^g$ are all the constant vectors, therefore $|X^g| = \ell^2-1$. If $r\neq 0$ and $s \neq 0$, then $X^g$ contains scalar multiplications of $(1, s, \cdots, s^{\ell-1})$ if $s^{\ell} = 1$ (i.e., $s\equiv 1\mod \ell$) and $X^g$ is empty if $s^{\ell}\neq 1$. Therefore summing up all terms, we obtain that the number of orbits within $N$ is
	\begin{equation}	
 \frac{1}{\ell^2(\ell-1)} \Big(\ell^{2\ell}-1 + \ell\cdot(\ell^{\ell}-1) + (\ell-1) \cdot \ell  \cdot (\ell^2-1)\Big).
	\end{equation}

    Now we compute $b(\pi, \phi)$ where $\phi: G_{\Q} \to C_{\ell}$ be corresponding to the unique degree $\ell$ sub-extension $F/\Q$ in $\Q(\mu_{\ell^2})/\Q$. To count the number of orbits with Burnside's Lemma, we have $X = N\backslash \{ e\}$ acted by $\Gal(\Q(\mu_{\ell^2})/\Q)$. Notice that the number of orbits is at least 
    \begin{equation}
    	\frac{|X^e|}{\ell(\ell-1)}= \frac{\ell^{2\ell}-1}{\ell(\ell-1)}. 
    \end{equation}
    
    When $\ell\ge 3$, we have
$$b(\pi, \phi) > b_M(G^{\text{rad}}, \Q).$$
\end{proof}

\begin{remark}\label{rmk:b=b}
	We remark that by a simple group theoretic consideration: in Lemma \ref{lem:Cl-wr-Cl-1}, 
	\begin{equation}
		b(\pi, \phi)= b_M(C_{\ell}, \Q(\phi)),
	\end{equation}
	where $\Q(\phi)$ is the unique $\Z/m\Z$-quotient of $\Q(\mu_{\ell})$. And in Lemma \ref{lem:Cl2-wr-Cl},
	\begin{equation}
		b(\pi, \phi)= b_{M}(C_{\ell^2}^{\text{rad}}, \Q(\phi)).
	\end{equation}
	where $\Q(\phi)$ is the unique $\Z/\ell\Z$-quotient of $\Q(\mu_{\ell^2})$.
\end{remark}

We now provide a couple examples with explicit numbers being computed. 

\begin{example}\label{exm:C3wrC4-bt}
Let $G = C_3\wr C_4$ and consider $\text{rad}$. There is a unique $C_2$ quotient of $G$. For each $\phi: \Gal(\Z/d\Z)^{\times} \to \Z/2\Z$, $\pi$ is unique, so we have
	\begin{itemize}
			\item 
		$\Q(i)$: $b(\pi,\phi) = 17$,
		\item 
		$\Q(\sqrt{3})$: $b(\pi, \phi) = 17$,
		\item 
	$\Q(\mu_3)$: $b(\pi, \phi) = 29$,
	\item 
	$\Q$: $b(\pi, \phi) = 19$.
	\end{itemize}
	Therefore 
	$$b_T(G^{\text{rad}}, \Q) = 29, \quad \quad b_M(G^{\text{rad}}, \Q) = 19.$$
	Since both $\Q(\mu_3)$, $\Q(i)$ and $\Q(\sqrt{3})$ cannot be embedded to a $C_4$-extension, by Conjecture \ref{conj:refined-malle}, we expect $b(G^{\text{rad}}, \Q) = b_M(G^{\text{rad}}, \Q)$. 
\end{example}

We now make two remarks for Example \ref{exm:C3wrC4-bt}. Firstly, we can see that as the intermediate field becomes larger, the $b$-constant can become bigger or smaller. This demonstrate that the comparison and simplification we did in Theorem \ref{thm:bt-simple} is optimal in general. Secondly, the largest $b$-constant appears when the intermediate field is $\Q(\mu_3)$ and we encounter failure in lifting this extension to a $G$-extension. This means that for product of ramified primes, there will also be counterexamples against $b_T$ due to lifting problem. 

It is conjectured in \cite[Conjecture 5.1]{koymans2023malle} that $b(\pi, \phi)$ is correct for $N_{Q,\pi, \phi}(G^{\text{rad}},X)$ for nilpotent groups. We now give the following example showing that counterexamples due to lifting problem can also appear when we restrict to $p$-groups.
\begin{example}\label{exa:koymans-pagano}
Let $G = C_4\wr C_4$. This group has exponent $16$ and $G^{ab} = C_4\times C_4$. There are altogether $26$ $b(\pi, \phi)$ to look at. Moreover, for each $\phi$, there are possibly more than one $\pi$ that we need to look at and they may correspond to different $b$-value, e.g., let $\phi$ correspond to $\Q(\mu_{16})$, there are three surjections from $G$ to $C_2\times C_4$, 
\begin{equation}
b(\pi_1, \phi)= 21,\quad \quad b(\pi_2, \phi)= 23, \quad\quad b(\pi_3, \phi)= 23.
\end{equation}
The largest $b$ is
\begin{equation}
    b(\pi, \phi)= 79, \quad \text{where} \quad \Q(\phi) = \Q(i), \quad \Ker(\pi) = C_4^2\rtimes C_2 \subset C_4\wr C_4,
\end{equation}
Notice that $\Q(i)$ is totally imaginary, therefore cannot be embedded into a $C_4$-extension, this means that $N_{Q,\pi,\phi}(G^{\text{rad}},X)=0$ for this pair. 
\end{example}

\subsection{Field Counting: Number Fields}
In this subsection, we give statements on $b(G, \Q)$ and $b(G^{rad}, \Q)$ in Theorem \ref{thm:example-1}, Theorem \ref{thm:example-rad-1} and Theorem \ref{thm:example-rad-2}. 

\begin{lemma}
Let $\ell$ be an odd prime number and $d = \prod_i p_i^{r_i} \neq 2$ where $p_i$ are all prime numbers. Let $G = C_{\ell} \wr C_{d}\subset S_{\ell d}$ and $\gcd(q, |G|) =1$. Denote $\gcd(d, \ell-1) = \prod_i p_i^{s_i}$. Let $s=\text{val}_2(\ell-1) -1$ when $\text{val}_2(d) > \text{val}_2(\ell-1)$ and $s = 0$ otherwise. Then
	$$b(G, \Q) = \prod_{i, r_i = s_i} p_i^{s_i} \cdot 2^s .$$
\end{lemma}
\begin{proof}[Proof of Theorem \ref{thm:example-1}, over $\Q$]
	By \cite[Corollary 1.6]{ALOWW} for $d>2$ and $G = C_{\ell} \wr C_d$ the inequality is satisfied as
	$$\frac{1}{2} + \frac{p}{d(p-1)} < \frac{\ell}{\ell-1},$$
	where $p$ is the minimal prime divisor of $d$. Then we have
	\begin{equation}
		b(G, \Q) = \max_{F/\Q, \Gal(F/\Q) = C_d} b(C_{\ell}, F).
	\end{equation}
	For $G = C_{\ell}$, by \cite{Wri89}, $b(F, C_{\ell}) = b_M(F, C_{\ell}) = [F\cap \Q(\mu_{\ell}):\Q]$. 
	
	Recall the notation $\gcd(d, \ell-1) = \prod_i p_i^{s_i}$ and $d = \prod_i p_i^{r_i} \neq 2$, and denote $\ell-1 = \prod_i p_i^{u_i}$, that is, $s_i = \min\{ r_i, u_i \}$. Since $\Q(\mu_{\ell})$ is cyclic over $\Q$, for each $n| \gcd(d, \ell-1)$, there exists a unique cyclotomic subfield $M = M_n\subset \Q(\mu_{\ell})$ that is only ramified at $\ell$ with degree $n$ over $\Q$. Given $n = \prod_i p_i^{t_i}$ with $0\le t_i\le s_i$, it follows from Theorem \ref{thm:embedding-number-field} that $M_n$ can be embedded into a $C_d$-extension if and only if 1) $\ell \equiv 1 \mod p_i^{r_i}$ for each $p_i|n$, and 2) if $M_n$ is totally imaginary, then $\text{val}_2(d) = \text{val}_2(n)$. Notice that $\ell \equiv 1 \mod p_i^{u_i}$, the first condition amounts to saying that $u_i\ge r_i$ whenever $t_i>0$, equivalently, $s_i = r_i$. For the second condition, $M_n$ is totally imaginary iff $n\nmid (\ell-1)/2$ iff $\text{val}_2(n) = \text{val}_2(\ell-1)$. Thus in this case, we require $\text{val}_2(d) = \text{val}_2(n) = \text{val}_2(\ell-1)$, i.e., $r_i = t_i = u_i$. Therefore the maximal $n$ where $M_n$ can be embedded into a $C_d$-extension can be described by specifying $t_i$: at odd primes, if $u_i \ge r_i$ (i.e., $r_i = s_i$), then let $t_i = s_i = r_i$, otherwise $0$; at $p=2$, if $r_i = u_i$, then we let $t_i = r_i = u_i = s_i$, if $r_i<u_i$, then we let $t_i = s_i = r_i$, if $r_i>u_i$, then we let $t_i = s_i-1 = u_i-1$. It then follows that $b(T, \Q)$ is this particular $n$, which is $\prod_{r_i = s_i } p_i^{s_i} \cdot 2^{s}$ where $s=\text{val}_2(\ell-1) -1$ only when $\text{val}_2(d) > \text{val}_2(\ell-1)$.
\end{proof}

Next we compute a lower bound for $b(G^{\text{rad}}, \Q)$ for $G = C_{\ell} \wr C_{d}$ and $C_{\ell^2} \wr C_{\ell}$.
\begin{lemma}
For $G = C_{\ell} \wr C_{d}\subset S_{\ell d}$ with respect to $N= C_{\ell}^{d}$ and $G = C_{\ell^2} \wr C_{\ell}\subset S_{\ell^3}$ with respect to $N = C_{\ell^2}^{\ell}$. We show that $$b(G^{\text{rad}}, \Q) \ge b(\pi,\phi),$$
where $\phi$ correponds to the unique cyclotomic $C_d$ subfield of $\Q(\mu_{\ell})$ and $C_{\ell}$-subfield of $\Q(\mu_{\ell^2})$.
\end{lemma}
\begin{proof}
By taking $T = N$, this can be translated to \cite{alberts2021harmonic,alberts2022statistics} and choosing any $\pi: G_{\Q} \to G$ such that its natural restriction $\pi: G_{\Q} \to G/N$ corresponds to the cyclotomic extension $F = \Q(\mu_{\ell})$ (respectively the unique $C_{\ell}$-subfield $F$ contained in $\Q(\mu_{\ell^2})$). Since they are wreath product, there exists $G$ extensions containing $F$. We then obtain that the number of $G$-extension containing $F$ with $\text{rad}< X$ has an asymptotic distribution with $a(G^{\text{inv}}) = 1$ and $b(\pi, \phi)$. This gives a lower bound on $b(G^{\text{rad}}, \Q)$ already. 
\end{proof}

\subsection{Field Counting: Global Function Fields}\label{ssec:counting-function-field}
Our main task in this section is to give a proof for the following lemma to prove Theorem \ref{thm:example-1}:
\begin{lemma}\label{lem:counting-function-field-Cl-wr-Cd}
	Let $\ell$ be an odd prime number and $d = \prod_i p_i^{r_i} \neq 2$ where $p_i$ are all prime numbers. Let $G = C_{\ell} \wr C_{d}\subset S_{\ell d}$ and $\gcd(q, |G|) =1$, $\Gal(\F_q(t)(\mu_{\ell})/\F_q(t)) = \Gal(\Q(\mu_{\ell})/\Q)$ and $q$ large enough comparing to $G$. Denote $\gcd(d, \ell-1) = \prod_i p_i^{s_i}$.
	$$b(G, \F_q(t)) = b_T(G, \F_q(t)) = \prod_i p_i^{s_i}.$$
\end{lemma}

Analogous with the number field setting, for each $X$, we denote $N_{\F_q(t)}(G, X)$ to be the number of all $G$-extensions over $\F_q(t)$ with discriminant bounded by $X$. For us, in order to complete parallel statements with number fields, the discriminant of a $G$-extension $K/\F_q(t)$ is the product of local discriminant over all primes away from the chosen infinity place. We will emphasize on the differences for counting those extensions over global function fields, from that over number fields in this section, as the version over number fields is proved in \cite{ALOWW}.

\begin{remark}\label{rmk:theta-function-field}
We cannot state an asymptotic distribution for $N_{\F_q(t)}(G, q^m)$ in general over function field. Given the fact that all discriminant now is $q^m$, a precise asymptotic does not exist whenever we get the non-existence (e.g. $G = C_3$ over $\F_q(t)$). This issue persists no matter if we choose to count extensions with discriminant bounded by $q^m$ or exactly equal to $q^m$. Due to this reason, we state the corresponding Malle's conjecture over function fields in Conjecture \ref{conj:Malle-function-field} with $\Theta(\cdot)$ instead of $\sim$.
\end{remark}

Before we start the proof, we first give the following reformulation of Wright's result \cite{Wri89} for abelian extensions when the abelian group $A$ has $(|A|,q)=1$. Notice that \cite[Theorem I.3]{Wri89} only states a certain weighted partial sum for abelian extensions. We now show it implies Conjecture \ref{conj:Malle-function-field} for abelian groups $A$.
\begin{theorem}\label{thm:abelian-function-field}
	Conjecture \ref{conj:Malle-function-field} holds for abelian group $A$ when $(|A|, q) = 1$. 
\end{theorem}

\begin{proof}
Let $Q$ be a global function field and $a_{q^m}$ denotes the number of $A$-extensions over $k$ with discriminant exactly equal $q^m$. Recall that Wright has proved the statement that
	\begin{equation}
		\sum_{j=0}^{b-1} a_{q^{m+j}} \cdot q^{-j/a} \sim C (q^{m})^{1/a} \cdot P_{b-1}(m) + O(q^{m})^{1/a-\delta},
	\end{equation}
	where $a = a(A)$ and $b = b_M(A,Q)$, and $P_{b-1}$ is a polynomial with degree $b-1$ and $\delta>0$ is a small positive number. Summing the equation from $1$ to $m$ and rearranging the terms, we obtain
	\begin{equation}
		N_Q(A, q^m)\cdot \sum_j q^{-j/a} + \sum_{1\le j\le b-1} \sum_{1\le k\le j} a_{q^{m+k}} \cdot q^{-(b-j)/a}  \sim C \sum_{1\le i\le m} (q^{i})^{1/a} \cdot P_{b-1}(i) + O(q^{m})^{1/a-\delta}.
	\end{equation}
	Notice that the summation over $i$ on the right hand side is also $q^{m/a} \cdot \tilde{P}_{b-1}(m)$ with $\tilde{P}$ a degree $b-1$ polynomial. It immediately follows that $N_Q(A, q^m) \le C_2 q^{m/a} m^{b-1}$ for some $C_2$. 
	
	To see the lower bound, notice that
	\begin{equation}
		N_Q(A, q^m)\cdot \sum_j q^{-j/a} \ge C \sum_{1\le i\le m-b+1} (q^{i})^{1/a} \cdot P_{b-1}(i) + O(q^{m})^{1/a-\delta}.
	\end{equation}
	which implies that there exists $C_1$ such that
	\begin{equation}
		N_Q(A, q^m) \ge C_1 q^{m/a} m^{b-1}.
	\end{equation}
\end{proof}

To determine $b(G, \F_q(t))$ for $G = C_{\ell} \wr C_d\subset S_{\ell d}$, we follow the similar idea in \cite{ALOWW}. The plan is to apply Theorem \ref{thm:abelian-function-field} to $G = C_{\ell}$ over any $C_d$-extension  $F/\F_q(t)$. For each $F$ and corresponding $\phi$, by Remark \ref{rmk:b=b}, we have $b(\pi, \phi) = b_M(C_{\ell}, F) = b(C_{\ell}, F)$. Therefore it suffices to prove that: 
\begin{equation}\label{eqn:function-field-induction}
	b(G, \F_q(t)) = \max_{F, \Gal(F/\F_q(t)) = C_d} b(C_{\ell}, F) ,
\end{equation}
and
\begin{equation}\label{eqn:embedding}
	\max_{F, \Gal(F/\F_q(t)) = C_d}  b(C_{\ell}, F)  = \max_{\phi} \quad b(\pi, \phi) = b_T(G, \F_q(t)),
\end{equation}
where $\phi$ varies over all $G_{\F_q(t)} \to C_d$. In (\ref{eqn:embedding}), the first equality is tautological, and the second equality follows from Theorem \ref{thm:embedding-function-field}: indeed, all abelian quotient $G/N$ with $\ind(N) = \ind(G)$ must contain $N = C_{\ell}^d$, and if $\bar{\phi}$ with smaller quotient can be lifted to $\phi$ with $C_d$-quotient, then by Lemma \ref{lem:bphi-comparison}, it suffices to check all $C_d$-quotient. 
\begin{remark}
Over number fields, the second equality in (\ref{eqn:embedding}) does not hold. This is the exact reason leading to the difference of $b$-constants over function fields and number fields in this paper. See Section \ref{sec:embed} and Theorem \ref{thm:embedding-number-field} for more discussion. 
\end{remark}

Finally, it remains to prove (\ref{eqn:function-field-induction}), that is, the analogue of the inductive argument in \cite{ALOWW} over function field. We sketch the idea as following. For each $C_d$-extension, we need to count $N_F(C_{\ell},X$ over a general $C_d$-extension $F/\F_q(t)$ with the predicted $a$ and $b$-constants, in addition to an upper bound on $N_F(C_{\ell},F)$ with a uniform dependence on $\Disc(F)$ that is small. Finally, adding up all $N_F(C_{\ell}, X)$ with the uniform dependence to show the total counting satisfy (\ref{eqn:function-field-induction}). The uniform dependence on counting abelian extensions are exactly characterized by the size of the $\ell$-torsion in class groups of $F$. Over function fields, this connection still holds. The $\ell$-torsion of the class group for a function field $F/\F_q(t)$ has a trivial bound $\ell^{2g}$ where $g$ is the genus.
\begin{remark}
It is well-known that $\ell$-torsion in class groups over number fields has a trivial bound $\Disc(k)^{1/2+\epsilon}$. Although it is conjectured to be bounded by $\Disc(k)^{\epsilon}$, current best results provide a bound at $\Disc(k)^{1/2-\delta}$ with some small $\delta>0$ in general. See \cite{ALOWW} for relevant references. The discriminant of a function field is linearly related to genus $g$ via Riemann Hurwitz formula. When $q$ becomes larger and larger, the trivial bound over function field behaves better and better and is less than $\Disc(k)^{\epsilon}$ for arbitrary small $\epsilon$, as long as $q$ is large enough with respect to $\epsilon$. This means that most inequalities in \cite{ALOWW} can be met when $q$ is large enough. This leads to many more cases of Malle's conjecture being proved over function field. 
\end{remark}

Before we give the proof, let's first give a summary of class field theory over function fields. For us, given a finite function field $F/\F_q(t)$, we obtain the ring of integers $O_F$ that contains all elements that are integral over all finite places (not including the places above $\infty$), which is the integral closure of $\F_q[t]$ inside $F$. We define the ideal class group of $F$ to be the class group of the ring $O_F$. It is a classical theorem that the ideal class group is also finite for function fields. But unlike number fields, $F$ can have many infinite unramified extensions, which is just the constant finite field extension. Relating the ideal class group with id\'ele class group $C_F$, we obtain the following 

\begin{tikzcd}	
	&\prod_v O_v^{\times} \times F_{\infty}^{\times}  \arrow{r}&   C_F:= \prod_v^{\prime} (F_v^{\times}, O_v^{\times}) \times F_{\infty}^{\times} \arrow{r}/ F^{\times} &  \Cl_F \arrow{r} & 0,.
\end{tikzcd}

\noindent It follows that
\begin{theorem}
	The ideal class group of a global function field is isomorphic to the Galois group of the maximal abelian unramified extension that is split at all places above infinity. 
\end{theorem}
\noindent On the other hand, by \cite[Proposition 14.1]{rosen2013number}, with $S = S_{\infty}$, we can relate the class group to the Jacobian of the curve $C_F$ corresponding to $F$ via
\begin{equation}
	 \Cl^0_{F}\to \Cl_F \to \Z/(d/i)\Z \to 0, 
\end{equation}
where $d = \gcd_{v|\infty} \{ \deg(v) \}$ and $i = \gcd_{v} \{ \deg(v) \}$, and $\Cl^0_{F} = \text{Jac}(C_F)(\F_q)$ is the Picard group. This means that we can bound $\lvert\Cl_F[\ell]\rvert\le O(\ell^{2g})$ by the structure of the Jacobian. 
\begin{lemma}\label{lem:uniformity-function-field}
	For any finite function field $F/\F_q(t)$, we have $\lvert \Cl_F[\ell]\rvert\le O_{\ell}(\ell^{2g})$.
\end{lemma}
\begin{lemma}
	Let $Q$ be a finite extension over $\F_q(t)$ and $\ell$ be a prime that is relatively prime to $q$. We then have
	$$N_Q(C_{\ell}, X) = O(C(\ell)^{2g}) X^{1/a(C_{\ell})} \ln^{b(C_{\ell}, Q)-1} X,$$
	where the constant $C(\ell)$ only depends on $\ell$.
\end{lemma}
\begin{proof}
	Denote $a = a(C_{\ell})$ and $b = b(C_{\ell}, Q)$ for short in the proof. 
	It follows from class field theory over function field that 
	$$N_k(C_{\ell},q^m) \le O(\ell^{2g(k)})\cdot |\Hom_{\le X}(\prod_{v} O_v^{\times}, C_{\ell})|,$$
	where $\Hom_{\le X}(\prod_{v} O_v^{\times}, C_{\ell})$ denotes the number of continuous homomorphisms from $\prod_v O_v^{\times}$ to $C_{\ell}$ with bounded discriminant. It has a generating series which is usually called Malle-Bhargava series. The series is an Euler product and can be compared to standard zeta functions
	\begin{equation}
		f(s):= \prod_{|v|\equiv 1\mod \ell} (1+ (\ell-1)|v|^{-(\ell-1)s}) = H_{Q}(s) \cdot \zeta_{Q(\mu_{\ell})}((\ell-1)s)^{b(C_{\ell}, F)}.
	\end{equation}
	Here $H_Q(s)$ is a holomorphic factor that is uniformly converging at $\Re(s)>1/a-\epsilon$ for some small $\epsilon>0$. Now letting $u = q^{-s}$, we define $g(u) = f(u(s)) =  \sum_{m} a_m u^m$. Since $\zeta_{Q(\mu_{\ell})}(s)$ is meromorphic except at its poles, and $H_Q(s)$ is holomorphic at $\Re(s)>1/a-\epsilon$, the complex function $g(u)$ is holomorphic within the disc $|u|< q^{-(1/(\ell-1))}$ and $g(u)/u^{m+1}$ is holomorphic everywhere in the disc except at $u=0$. We then have for small $0<\delta< \epsilon$ that
	\begin{equation}\label{eqn:Perron}
		2\pi i\cdot	a_m = \int_{|u|= q^{-1/(\ell-1) -\delta}} \dfrac{g(u)}{u^{m+1}} du = \int_{|u| = q^{-1/(\ell-1)+\delta}}  \dfrac{g(u)}{u^{m+1}} du- \sum_{u_0^{\ell-1} = q^{-1}} \Res_{u= u_0} \dfrac{g(u)}{u^{m+1}},
	\end{equation}
	by shifting the contour integration from the smaller circle to the larger circle. The only possible poles for $f(s)$ with $\Re(s) > 1/a-\epsilon$ are at $s$ where $q^{as }= q$, i.e., $s_k = 1/a+ 2\pi i/a\log q \cdot k $, therefore the possible poles for $g(u)$ are $u(s_k) = q^{-1/a} \cdot \mu_{a}^k$. 
	
	We now estimate the terms in (\ref{eqn:Perron}). The residue of $g(u)/u^{m+1}$ at $u =u(s_k)$ will serve as the main term for $a_m$. We add $-1$ since $C_1$ reversed direction in the second contour integral. 
	\begin{equation}
		\begin{aligned}
		\sum_{u_0^{\ell-1} = q^{-1}} \Res_{u= u_0} \dfrac{g(u)}{u^{m+1}} =(\ln q)^b \frac{(m+b-1)!}{(b-1)!m!} u(s_0)^{-m} \cdot \sum_{k \mod a} \Res(s=s_k) \mu_{a}^{-km},
		\end{aligned}
	\end{equation}
	where $u(s_0) = q^{-1/a}$ and $(m+b-1)!/m!$ is a polynomial of degree $b-1$ in terms of $m$. The integration can be bounded by 
	\begin{equation}
		\max_{|u| = R = q^{-1/a +\delta}} |g(u)|  \cdot R^{-(m+1)} \cdot 2\pi R = O\Big(\max_{|u| = R = q^{-1/a +\delta}} |g(u)| \Big) \cdot q^{(1/a-\delta)m},
	\end{equation}
	which is power-saving from the leading terms. The dependence on the base field $Q$ is from $\max_{|u| = R = q^{-1/a +\delta}} |g(u)|$. Here $g(u) = \zeta_{Q(\mu_{\ell})}(as)^{b}$. By Weil's Conjecture, we have
	$$\zeta_{Q(\mu_{\ell})} = \dfrac{ \prod_{1\le i\le 2g} (1 - \alpha_i q^{-s})}{(1-q^{-s})(1-q^{1-s})},$$
	where $|\alpha_i| = \sqrt{q}$ and $g$ is the genus. Therefore on $|u| = q^{1/a+\delta}$ we have
	$$|g(u)|\le O(2^{2g}).$$
	Combining above, we obtain that there exists $\epsilon>0$ such that 
	\begin{equation}
		a_m = O(q^{m/a} P(m)) + O(2^{2g} q^{m/a-\epsilon}),
	\end{equation}
 where the polynomial $P(m)$ has degree at most $b-1$ in $m$ and the constants can be understood by expanding the rational polynomial. Via an explicit computation, we have
	$$\zeta_{Q(\mu_{\ell})}(s)\sim \prod (1-\alpha_j q^{-s}) \cdot \dfrac{ s-1}{(q^{1-s}-1)(1-q^{-s})} = \sum b_i (s-1)^i \cdot \sum a_j (s-1)^j,$$
	where $a_j$ and $b_i$ serve as the coefficients for the Taylor expansion around $s = 1$. The $a_j$ does not depend on $Q$. The polynomial $P(m)$ depends on the first $b$ coefficients of $\zeta_{Q(\mu_{\ell})}$. We can compute coefficients $b_j$ via taking derivatives of the $\zeta$-polynomials directly, as an example,
	$$b_0 = \prod (1-\alpha_j q^{-1}), \quad \quad b_1 = b_0 \cdot \sum \alpha_j\ln q \cdot q^{-1} (1-\alpha_jq^{-1})^{-1}.$$
	Similarly, we can compute $b_j$ and obtain that for $0\le j \le b$
	$$b_j= O(C^g)$$
	for some constants $C$ only depending on $b$ and $g$. This finishes the proof that $a_m = O(C^g)q^{m/a} m^{b-1}$. The statement follows by adding up over all $q^m\le X$.
\end{proof}

\begin{proof}[Proof of Lemma \ref{lem:counting-function-field-Cl-wr-Cd}]
	Firstly, it is easy to compute that $b_T(G, \F_q(t)) = \gcd(d, \ell-1)$  by Theorem \ref{thm:bt-simple} by only considering $\phi$ exactly corresponding to cyclotomic subfields of $\F_q(t)(\mu_{\ell})$.
	
	Next in order to show it is the true count, it suffices to prove the equality (\ref{eqn:function-field-induction}) from previous discussion, i.e., give the inductive argument for this case. 
	
	We need to firstly show that the number of $C_{\ell}$-extensions $L/F$ with $\Gal(L/\F_q(t)) = C_{\ell} \wr C_d$ for each $C_{d}$-extension $F/\F_q(t)$ is bounded from above and below by $X^{1/a(C_{\ell})} \ln^{b(C_{\ell}, F)-1} X$. It is clear from Theorem \ref{thm:abelian-function-field} that the upper bound holds. Let $v$ be a place in $\F_q(t)$ that becomes split in $F$, and $v = \prod_i w_i$. The number of $C_{\ell}$-extensions that are ramified at $w_1$ and unramified at all other $w_i$ can be counted by \cite[Theorem 7.2]{Wri89} (together with the remarks after Theorem $7.3$), and has a leading term with the same order with the total counting. All extensions satisfying this local condition have total Galois group $C_{\ell} \wr C_d$, since the Frobenius at $v$ generate $C_{\ell}^d$ in $C_{\ell} \wr C_d$. This implies that the lower bound for $N(G, X)$ is $X^{1/a} \ln^{b-1} X$ where $b = \max_F b(C_{\ell}, F)$. 
	
	Secondly, we need to show that we can add up each counting over all $F$. With the uniformity proven in Lemma \ref{lem:uniformity-function-field}, we have that
	$$N(G, X) \le  \sum_{F} N_F(C_{\ell}, X/\Disc(F)^{\ell}) = \sum_{F} O((C)^{g(F)} \Disc(F)^{-\ell}) X^{1/a} \ln^{b-1} X,$$
	where 
	$$O\left(\sum_{F} C^{g(F)} q^{-\ell(2g-2)}\right) = O(1)$$
	as long as $q$ is large enough comparing to $\ell$. Therefore we prove $N(G, X)$ has an upper bound and lower bound with the same order
	$$a = a(C_{\ell}), \quad\quad b= \max_{F} b(C_{\ell}, F). $$
\end{proof}

\section{Embedding Cyclotomic Extensions}\label{sec:embed}
In this section, we would like to consider Problem \ref{que:embedding-cyclotomic}. Precisely, given a surjective group homomorphism $\pi: G\to B$ and a cyclotomic $B$-extension $F/Q$ (equivalently a surjective continous group homomorphism $\phi: G_{Q} \to B$), we say the embedding problem $\mathscr{E}(G_Q, \phi, \pi)$ is \emph{solvable} if there exists a continuous group homomorphism $\tilde{\phi}$ so that the following diagram commutes. We say $\mathscr{E}(G_Q, \phi, \pi)$ is \emph{properly solvable} if moreover $\tilde{\phi}$ is surjective. 
\begin{center}\label{diag:embedding-3}
	\begin{tikzcd}	
		&  &  & G_Q \arrow[ld, "\tilde{\phi}" ', dashed]\arrow[d,"\phi",two heads] &  \\		
		0\arrow{r} & \Ker(\pi) \arrow{r} & G \arrow[r, "\pi"] & B \arrow{r} & 0 \\
	\end{tikzcd}
\end{center}

We are going to see that the crucial difference between global function fields and number fields is that $\Gal(Q^{cyc}/Q)$ is projective for function fields, but not for number fields! Theorem \ref{thm:embedding-function-field} and Theorem \ref{thm:embedding-number-field} serves the purpose to give a flavor of this question on two sides. 


\subsection{Function Fields}
Over function field, all cyclotomic embedding problem has solutions, due to the fact that the constant extension is projective. However, it is not known in general that ``solvable implies properly solvable''. The existence of a proper solution of such cyclotomic embedding problem can be considered as a decorated inverse Galois problem over global function field. We now give many cases when $G$ is solvable, based on extensive results on inverse Galois problem for solvable groups. 
\begin{theorem}\label{thm:embedding-function-field}
	Let $Q$ be a global function field with maximal constant field $\F_q$. For any cyclotomic $B$-extension $\phi: G_Q\twoheadrightarrow B$ and a surjective group homomorphism $\pi: G\to B$, the embedding problem $\mathscr{E}(G_Q,\phi, \pi)$ is always solvable. If $(\Ker(\phi), q(q-1)) = 1$ and $\Ker(\phi)$ is solvable, then $\mathscr{E}(G_Q,\phi, \pi)$ is always properly solvable.
\end{theorem}
\begin{proof}
	Firstly, by \cite[Corollary $(9.5.8)$]{NSW}, it suffices to show that $\mathscr{E}(G_Q,\phi, \pi)$ is solvable. We are going to show that it is always solvable. Notice that we can decompose $\phi = \phi_1\circ \phi_0$ as a composition of $\phi_0$ and $\phi_1$, therefore it suffices to show that $\mathscr{E}(G^{cyc}_Q, \phi_1, \pi)$ is solvable.
	\begin{center}\label{diag:embedding-function-field}
		\begin{tikzcd}	
			&  &  & G_Q \arrow[d,"\phi_0",two heads] &  \\		
			&  &  & G^{cyc}_Q \arrow[ld, "\tilde{\phi_1}" ', dashed]\arrow[d,"\phi_1",two heads] &  \\		
			0\arrow{r} & \Ker(\pi) \arrow{r} & G \arrow[r, "\pi"] & B \arrow{r} & 0 \\
		\end{tikzcd}
	\end{center}
The group $G^{cyc}_Q\simeq \hat{\Z}$ is a projective profinite group, therefore $\mathscr{E}(G^{cyc}_Q, \phi_1, \pi)$ is always solvable.
\end{proof}

\subsection{Number Fields}
What makes the key difference between function fields and number fields is that $G^{cyc}_{\Q}$ is no longer projective as a profinite group, therefore it is not always solvable $\mathscr{E}(G_Q,\phi, \pi)$.

We firstly give an analogue of Theorem \ref{thm:embedding-function-field}. The analogue of constant extension from function field on the number field side is the $\Z_p$-cyclotomic tower. This leads to:
\begin{theorem}
Given the embedding problem $\mathscr{E}(G_k,\phi, \pi)$ where $k$ is a number field. Assume $\phi:G_k \to B$ cut out a subfield $k(\phi)$ of the unique $\hat{\Z}$-extension contained in $k(\mu_{\infty})$, and $\Ker(\pi)$ is solvable with $(\mu(k),|\Ker(\pi)|) =1$. Then the problem $\mathscr{E}(G_{\Q},\phi, \pi)$ is properly solvable.
\end{theorem} 
\begin{proof}
Firstly, by \cite[Corollary $(9.5.8)$]{NSW}, it suffices to show that $\mathscr{E}(G_Q,\phi, \pi)$ is solvable. Since $\phi$ factors through the unique $\hat{\Z}$-extension over $k$ and $\hat{\Z}$ is projective, it is always solvable. 
\end{proof}

Next, we give a criteria for $\mathscr{E}(G_\Q, \phi, \pi)$ when $G$ is abelian. It can also be generalized to some number fields. Notice that for each place $v$ of $k$, after fixing a prime above $v$ in $\bar{k}$, the map $\phi$ induces a local map $\phi_v: G_{k_v} \to G_k\to B$, and induces a local group extension $\pi_v$
\begin{equation}
    0 \to \Ker(\pi) \to \pi^{-1}(\phi_v(G_{k_v})) \to \phi_v(G_{k_v}) \to 0.
\end{equation}
We now obtain many local embedding problems $\mathscr{E}(G_{k_v},\phi_v, \pi_v)$. In most cases when $\Ker(\pi)$ is contained in the center of $G$, we have local-global principle, that is:
\begin{lemma}\label{lem:central-embedding}
Consider the embedding problem $\mathscr{E}(G_k,\phi, \pi)$ where $k$ is a global field. Assume $\Ker(\pi)$ is contained in the center of $G$. If one of the following conditions holds:
\begin{enumerate}
    \item 
    $k$ is global function field
    \item 
    $|\Ker(\pi)|$ is odd
    \item 
    Let $e$ be the exponent of $2$-primary part of $\Ker(\pi)$ and $e\nmid 8$.
    \item 
    $e|8$ and $\Gal(k(\mu_{e})/k) = D_{\mathfrak{p}}$ for some $\mathfrak{p}$ in $k$.
\end{enumerate}
then $\mathscr{E}(G_k,\phi, \pi)$ is solvable if and only if $\mathscr{E}(G_{k_v},\phi_v, \pi_v)$ is solvable for each $v$. 
\end{lemma}
\begin{proof}
    Firstly, for central embedding problem, solvable implies properly solvable since one can twist solutions by $H^1(k, \Ker(\pi))$, i.e. abelian extensions over $k$ with Galois group a subgroup of $\Ker(\pi)$, to make a surjective lifting. 

    Secondly, by \cite[Proposition 3.5.9]{NSW}, the embedding problem is solvable if and only if $inf(\epsilon) = 0 \in H^2(k, \Ker(\pi))$ where $\epsilon \in H^2(B, \Ker(\pi))$ corresponds to the group extension $\pi$. By the assumption, we know that $inf(\epsilon) = 0 \in H^2(k_v, \ker(\pi))$. This means that $inf(\epsilon) \in \Sha^2(k, \Ker(\pi))$. By Poitou-Tate duality we know that $\Sha^2(k, \Ker(\pi)) \simeq \Sha^1(k, \Ker(\pi)')^{\vee}$. If $\Ker(\pi) \simeq \prod_i \Z/{n_i}\Z$, then $\Ker(\pi) = \prod_i \mu_{n_i}$. By \cite[Theorem 9.1.9]{NSW}, we know that $\Sha^1(k, \prod_i\mu_{n_i}) = \prod_i \Sha^1(k, \mu_{n_i}) =0$ if we are not in the special case. By \cite[Remark 3, p.528]{NSW}, non special case is equivalent to the condition that $\Gal(k(\mu_{e})/k) = D_{\mathfrak{p}}$ for some $\mathfrak{p}$ in $k$. 
\end{proof}
\begin{remark}
The common references of such a local-global principle for central kernel is \cite[Corollary 10.2]{malle1999inverse}. It only includes the case when the kernel is elementary abelian.
\end{remark}

From Lemma \ref{lem:central-embedding}, we obtain the following theorem:
\begin{theorem}\label{thm:nice-k}
Let $k$ be a number field. If one of the following condition holds:
\begin{enumerate}
    \item 
     $\Gal(k(\mu_{2^{\infty}})/k) = \hat{\Z}_2$,
    \item 
    $\Gal(k(\mu_{2^{\infty}})/k) = \hat{\Z}_2 \times \Z/2\Z$ and there exists a prime $v$ above $2$ with $D_v = \Gal(k(\mu_{2^{\infty}})/k)$.
\end{enumerate}
then $\Sha^2(k, A) =0$ for all trivial module $A$, and all central embedding problem over $k$ satisfy local-global principle. 
\end{theorem}
Notice that latter condition just need to be checked at the smallest $r$ where $\Gal(k(\mu_{2^r})/k)\simeq C_2\times C_2$ by a Frattini argument applied to the decomposition group. Moreover the decomposition group $D_v$ is the full Galois group for one prime $v$ above $2$ holds exactly when it holds for all primes $v|2$ by Galois theory. 

Now notice that if a prime $v$ in unramified in $k(\phi)/k$, then locally $\mathscr{E}(\mu(\Q_v),\phi_v, \pi_v)$ is always solvable, we further simplify to:
\begin{theorem}\label{thm:embedding-number-field}
Let $G$ be an abelian group and $k$ be a number field satisfying Theorem \ref{thm:nice-k}. Then the problem $\mathscr{E}(G_{k},\phi, \pi)$ is properly solvable if and only if $\mathscr{E}(G_{k_v},\phi_v, \pi_v)$ is solvable at each ramified $v$. 
\end{theorem}

We do have $\Sha^1(\Q, \mu_{2^r}) = 0$ for all $r$, since $2$ has full decomposition group. Therefore central embedding problem satisfy local-global principle over $\Q$. 
\begin{example}\label{example:embedding}
Let $Q = \Q$, an odd prime $\ell$, and $d>0$. $\Sha^2(\Q, A)=0$ for all trivial modules $A$. Let $n| \gcd(d, \ell-1)$, and $B = C_n$ and $\phi: G_{\Q} \to C_n$ corresponds to the unique $C_n$-sub extension $F$ contained in $\Q(\mu_{\ell})$. By Theorem \ref{thm:embedding-number-field}, it suffices to check the local solvability by $ v= \ell$ and $v = \infty$. Notice that $(\ell, \ell-1) = 1$, the prime $\ell$ is totally tamely ramified in $F$. The local map $\phi_v$ then maps $\mu(\Q_{\ell}) = \{ \mu_{\ell-1} \} \simeq \Z/(\ell-1)\Z$ by sending the generator to generator of $C_{n}$. In order to lift $\phi$, the only way is to map the generator of $\mu(\Q_{\ell})$ to $g\in C_d$ such that $\bar{g}$ generate $C_n$ and $(\ell-1)g = 0$. Writing $d = \prod_i p_i^{r_i}$, $\ell-1 = \prod_i p_i^{u_i}$, $\gcd(d, \ell-1) = \prod_i p_i^{s_i}$ and $n = \prod_i p_i^{t_i}$. We consider all abelian groups here as a direct product of cyclic $p_i$-groups. If $\phi(g_i) = 1$ for every $p_i| n$, then we need $(\ell-1) g_i = 0$ for a generator $g_i$ in $C_d$, equivalently, this means that $u_i\ge r_i$. At $v = \infty$, if $\phi$ is totally real, i.e., $n| (\ell-1)/2$, then always solvable. If $n\nmid (\ell-1)/2$, then it suffices that $2g_i = 0$ and $\phi(g_i)$ has order $2$. This requires $2$-Sylow subgroup for $C_d$ and $C_n$ being the same.  
\end{example}

\section{Acknowledgement}
The author was partially supported by Foerster-Bernstein Fellowship at Duke University and NSF grant DMS-2201346. A significant part of this work is done during my graduate school, so I would like to thank my advisor Melanie Matchett Wood for many inspiring questions and helpful discussions on the topic, especially on understanding the statement of T\"urkelli's modification. I would like to thank Brandon Alberts, Jordan Ellenberg, J\"urgen Kl\"uners, Aaron Landesman, Robert J. Lemke Oliver, Yuan Liu and Bianca Viray for helpful conversations. I would also like to thank Peter Koymans, Aaron Landesman, Daniel Loughran, Gunter Malle, Even O'Dorney, and Tim Santens for helpful comments on a previous preprint.

\Addresses
\end{document}